\documentclass[reqno,11pt]{amsart}
\usepackage{newtxtext}
\usepackage{float,amssymb,bm,url}
\usepackage {tikz}
\usetikzlibrary{arrows,positioning}

\usetikzlibrary{cd,calc,positioning,angles,decorations.markings}
\tikzcdset{arrow style=tikz, diagrams={>=stealth}}

\tikzset{
sedge/.append style={shorten <=9pt, shorten >=9pt}
}
\makeatletter
\tikzset{
  @arc through/.style 2 args={
    to path={
      \pgfextra
        \pgfextract@process\pgf@tostart{\tikz@scan@one@point\pgfutil@firstofone(\tikztostart)\relax}%
        \pgfextract@process\pgf@tothrough{\tikz@scan@one@point\pgfutil@firstofone#1}%
        \pgfextract@process\pgf@totarget{\tikz@scan@one@point\pgfutil@firstofone(\tikztotarget)\relax}%
        \pgfextract@process\pgf@topointMidA{\pgfpointlineattime{.5}{\pgf@tostart}{\pgf@tothrough}}%
        \pgfextract@process\pgf@topointMidB{\pgfpointlineattime{.5}{\pgf@totarget}{\pgf@tothrough}}%
        \pgfextract@process\pgf@tocenter{%
          \pgfpointintersectionoflines{\pgf@topointMidA}
            {\pgfmathrotatepointaround{\pgf@tothrough}{\pgf@topointMidA}{90}}
            {\pgf@topointMidB}{\pgfmathrotatepointaround{\pgf@tothrough}{\pgf@topointMidB}{90}}}%
        \pgfcoordinate{arc through center}{\pgf@tocenter}%
        \pgfpointdiff{\pgf@tocenter}{\pgf@tostart}%
        \pgfmathveclen@{\pgfmath@tonumber\pgf@x}{\pgfmath@tonumber\pgf@y}%
        \edef\pgf@toradius{\pgfmathresult pt}
        \pgfmathanglebetweenpoints{\pgf@tocenter}{\pgf@tostart}%
        \let\pgf@tostartangle\pgfmathresult
        \pgfmathanglebetweenpoints{\pgf@tocenter}{\pgf@totarget}%
        \let\pgf@toendangle\pgfmathresult
        \ifdim\pgf@tostartangle pt>\pgf@toendangle pt\relax
          \pgfmathsetmacro\pgf@tostartangle{\pgf@tostartangle-360}%
        \fi
        #2%
          \pgfmathsetmacro\pgf@toendangle{\pgf@toendangle-360}%
        \fi
      \endpgfextra
      arc [radius=+\pgf@toradius, start angle=\pgf@tostartangle, end angle=\pgf@toendangle] \tikztonodes
    }},
  arc through ccw/.style={@arc through={#1}{\iffalse}},
  arc through cw/.style={@arc through={#1}{\iftrue}},
}
\makeatother

\tikzset{middlearrow/.style={
        decoration={markings,
            mark= at position 0.5 with {\arrow{#1}} ,
        },
        postaction={decorate}
    }
}

\tikzset{
vertex1/.style ={circle, inner sep=1.5pt, fill=white, draw},
vertex2/.style ={circle, inner sep=1.5pt, fill=black, draw},
every loop/.append style={min distance=12mm,shorten <=6pt, shorten >=6pt}
}

\newcommand{\Rbb}{\mathbb{R}}
\newcommand{\Zbb}{\mathbb{Z}}

\newcommand{\ud}{\,\mathrm{d}}
\newcommand{\p}{_{\ge0}}

\newcommand{\m}{^{-1}}
\newcommand{\argomento}{\operatorname{--}}

\usepackage{mathrsfs}
\newcommand{\ms}[1]{\mathscr{#1}}
\newcommand{\mc}[1]{\mathcal{#1}}
\newcommand{\mf}[1]{\mathfrak{#1}}
\newcommand{\mb}[1]{\mathbf{#1}}

\newcommand{\abs}[1]{\lvert#1\rvert}
\newcommand{\norm}[1]{\lVert#1\rVert}
\newcommand{\newword}[1]{\emph{#1}}
\newcommand{\angles}[1]{\langle #1 \rangle}
\newcommand{\floor}[1]{\lfloor #1 \rfloor}
\newcommand{\set}[1]{\{ #1 \}}
\newcommand{\cppvector}[2]{\bigl(\begin{smallmatrix}#1\\#2\end{smallmatrix}\bigr)}
\newcommand{\rppvector}[2]{(#1\;\;#2)}
\newcommand{\ppmatrix}[4]{\bigl(\begin{smallmatrix}#1&#2\\#3&#4\end{smallmatrix}\bigr)}

\DeclareMathSymbol{\upharpoonright}{\mathrel}{AMSa}{"16}
\let\restriction\upharpoonright

\DeclareMathOperator{\GL}{GL}

\DeclareMathOperator{\PP}{P}
\DeclareMathOperator{\PSL}{PSL}
\DeclareMathOperator{\SL}{SL}
\DeclareMathOperator{\SO}{SO}
\DeclareMathOperator{\Mat}{Mat}

\theoremstyle{plain}
\newtheorem{theorem}{Theorem}[section]
\newtheorem{lemma}[theorem]{Lemma}

\newtheorem{proposition}[theorem]{Proposition}

\theoremstyle{definition}
\newtheorem{definition}[theorem]{Definition}
\newtheorem{remark}[theorem]{Remark}
\newtheorem{example}[theorem]{Example}

\numberwithin{equation}{section}

\DeclareMathOperator{\Leb}{Leb}
\DeclareMathOperator{\Aff}{Aff}
\DeclareMathOperator{\length}{length}
\newcommand{\vertiii}[1]{{\left\vert\kern-0.25ex\left\vert\kern-0.25ex\left\vert #1 
    \right\vert\kern-0.25ex\right\vert\kern-0.25ex\right\vert}}

\begin{document}

\bibliographystyle{plain}

\sloppy

\title[Dual continued fractions]{Attractors of dual continued fractions}

\author[G.~Panti]{Giovanni Panti}
\address{Department of Mathematics, Computer Science and Physics\\
University of Udine\\
via delle Scienze 206\\
33100 Udine, Italy}
\email{giovanni.panti@uniud.it}

\begin{abstract}
Given a Farey-type map $F$ with full branches in the extended Hecke group $\Gamma_m^\pm$, its dual $F_{\sharp}$ results from constructing the natural extension of $F$, letting time go backwards, and projecting. Although numerical simulations may suggest otherwise, we show that the domain of $F_{\sharp}$ is always tame, that is, it always contains intervals. As a main technical tool we construct, for every $m=3,4,5,\ldots$, a homeomorphism $M_m$ that simultaneously linearizes all maps with branches in $\Gamma_m^\pm$, and show that the resulting dual linearized iterated function system satisfies the strong open set condition. We explicitly compute the
H\"older exponent of every $M_m$, generalizing Salem's results for the  Minkowski question mark function $M_3$.
\end{abstract}

\thanks{\emph{2020 Math.~Subj.~Class.}: 11A55, 37F32}
\thanks{The author is partially supported by the 
MIUR Grant E83C18000100006 \emph{Regular and stochastic behaviour in dynamical systems}.}

\maketitle

\section{Introduction}\label{ref1}

Continued fraction maps with branches in extended Hecke groups (``extended'' referring to making allowance for matrices of determinant $-1$) have a long history. The original and still widely used framework is that of Rosen fractions and their variants, which have been pivotal in contexts such as the determination of the set of cusps of Hecke groups, the study of the Markov spectrum, and the transfer operator approach to Selberg's zeta function; see ~\cite{rosen54}, \cite{Schmidtsheingorn95}, \cite{mayer_et_al12} and references therein.
Rosen fractions have infinitely many branches and strong ergodic properties, admitting finite invariant absolutely continuous measures. As a drawback, 
as Markov maps on an interval
they do not have full branches; this makes the analysis difficult, and a whole theory of \emph{matching} has been developed for dealing with the issue~\cite{bruin_et_al19}, \cite{kalle_et_al20}.

In this paper we are interested in c.f.~maps of the \emph{slow} type ---whose prototype is the Farey map--- which have finitely many branches, all of them full. In this setting parabolic cycles, absent in the Rosen fractions, are inevitable, and the ergodic properties are weaker, as only infinite invariant absolutely continuous measures are admissible. Nonetheless
slow maps over Hecke groups have been proved valuable in various settings, for example providing renormalization procedures for flows on translation surfaces,
and a transfer operator approach to Maass cusp forms~\cite{smillieulcigrai10},
\cite{pohl14}, \cite{panti20a}.

However the branches may be, the construction of the natural extension of a c.f. map $F$ remains an essential nontrivial step. Except in the simplest cases, it begins with computer simulation. One picks a point $(\omega,\alpha)$ and iterates $(\omega,\alpha)\mapsto(F(\omega),F_\omega\m(\alpha))$, where $F_\omega\m$ is the local inverse of $F$ at $\omega$. If the initial choice is ``sufficiently generic'', then the orbit distributes along the domain of a geometric model for the natural extension $F_e$ of $F$; one can then formulate conjectures about the shape of the domain, and possibly prove them. However, even for extremely simple maps (see, e.g., Figure~\ref{fig7}),
quite often the distribution of the orbit appears fractal, and one is lead to suspecting the presence of strange attractors, namely Cantor sets of positive Lebesgue measure.

As a main result of this paper we prove that, in the case of slow maps, these suspicions are unjustified: the domain of $F_e$ is always of the form
$I\times K$, where $I$ is the interval domain of $F$ and $K$ ---the attractor of the iterated function system determined by the inverse branches of $F$--- always contains a dense set of intervals. As a matter of fact, we prove that the above iterated function system satisfies the strong open set condition~\cite{schief94}, \cite{peres_et_al01}, \cite{lauraoye01}, and this yields
the existence ---and often the explicit determination---
of the dual map of $F$. The latter is the map that results by running the natural extension in reverse time, and projecting in the second coordinate; note that maps with nonfull branches may not have duals. 
Duals of slow maps have indeed occasionally popped up in the literature, often under a rather speedy treatment: ``take the transposed matrices''.

Our paper is structured as follows: 
in~\S\ref{ref2} we introduce the extended Hecke groups $\Gamma_m^\pm$, for $m=3,4,5,\ldots$, and define embeddings of a certain ``double tree'' monoid in $\Gamma_m^\pm$ and in the affine group of $\Rbb$ with coefficients in $\Zbb[1/(m-1)]$. In~\S\ref{ref5} we describe how finite subtrees of this double tree induce slow c.f.~maps, and state our main Theorem~\ref{ref8}; we prove all of this theorem under the assumption of its first key statement, Theorem~\ref{ref8}(1). In~\S\ref{ref10} we provide examples of dual maps, some of them of independent interest, and settle a question in~\cite{arnouxschmidt19} about the connectedness of the domain of the natural extension in the orientation-preserving case. In~\S\ref{ref11} we develop our main technical tool, namely the construction of a family of homeomorphisms $M_m:[0,(2\cos(\pi/m))\m]\to[0,1]$ that topologically conjugate the Hecke group action of the above double tree with its affine action. In particular $M_m$ ---which for $m=3$ is the classical Minkowski question mark function--- linearizes every slow map with branches in $\Gamma_m^\pm$. We explicitly determine ---using an extremal norm technique of independent interest--- the joint spectral radius of the family of matrices determining the ``easiest'' map with $\Gamma_m$ branches, and use it to compute the H\"older exponent of $M_m$; this generalizes Salem's determination of the exponent of $M_3$~\cite{salem43}. 
In the final section~\S\ref{ref12} we use the results in~\S\ref{ref11} and the divergence of the Poincar\'e series of $\Gamma_m^\pm$ along the monoid generated by the inverse branches of a slow map to complete the proof of Theorem~\ref{ref8}(1).

\section{Notation and preliminaries}\label{ref2}

Throughout this paper $m$ is an integer greater than or equal to $3$, and $\lambda=\lambda_m=2\cos(\pi/m)$. The \newword{extended Hecke group} $\Gamma^\pm$ is the subgroup of $\PSL^\pm_2\Rbb$ (the group of $2\times2$ matrices with real entries and determinant $\pm1$, up to sign change) generated by
\begin{equation}\label{eq1}
L=\begin{pmatrix}
1 & \\
\lambda & 1
\end{pmatrix},\qquad
S=\begin{pmatrix}
 & -1\\
1 & 
\end{pmatrix},\qquad
F=\begin{pmatrix}
 & 1\\
1 & 
\end{pmatrix}.
\end{equation}
Here and below blank spaces in matrices substitute zero entries.
Besides $\lambda,\Gamma^\pm,L$ defined above, in this paper we will introduce several objects depending on $m$, such as the monoid $\Sigma$ of Definition~\ref{ref20}, the measure $\eta$ of Theorem~\ref{ref8}, the Minkowski function~$M$ of~\S\ref{ref11}. In order to avoid burdening of notation we will drop the $m$ subscript whenever possible, occasionally reverting to an explicit $\lambda_5$, $M_3$, $\ldots$ when necessary for clarity.

We identify elements of $\PSL^\pm\Rbb$ with the maps they induce
on the upper-half plane $\mc{H}$ and its boundary $\PP^1\Rbb$: if $A=\ppmatrix{a}{b}{c}{d}$
(we are using a matrix to represent its projective class),
then $A(z)$ equals $(az+b)/(cz+d)$ if $A$ has positive determinant, and $(a\bar{z}+b)/(c\bar{z}+d)$
otherwise. We refer to~\cite{beardon95}, \cite{katok92} for basics of hyperbolic geometry.

The triangle in $\mc{H}$ of vertices $\zeta=\exp\bigl((1-1/m)\pi i\bigr),i,0$ is a fundamental domain for $\Gamma^\pm$, which is a copy of the abstract \newword{extended triangle group} $\Delta^\pm(2,m,\infty)$. The subgroup $\Gamma<\Gamma^\pm$ of matrices of determinant~$1$ is the \newword{Hecke group}, and has as fundamental domain the union of the above triangle with its $F$-image. For $j=1,2,\ldots,m-1$, we further introduce the matrices
\begin{equation}\label{eq2}
R=L\m S=\begin{pmatrix}
 & -1\\
 1 & \lambda
 \end{pmatrix},\qquad
A_j=SR^{-j},\qquad
A_f=F.
\end{equation}
The picture in Figure~\ref{fig1}, which is drawn for $m=7$, clarifies the situation. The matrix $R$ induces a clockwise rotation by $2\pi/m$ around~$\zeta$, while $S$ induces a rotation by $\pi$ around~$i$; note that although we compute in the upper-plane model, we draw pictures in the disk model. The extended Hecke group and the ordinary one are generated by
$S,R,F$ and $S,R$, respectively; as a matter of fact, $\Gamma$ is the free product of cyclic groups $\angles{S}*\angles{R}\simeq Z_2*Z_m$, while $\Gamma^\pm$ is the amalgamated product of dihedral groups $\angles{S,F}*_{\angles{F}}\angles{R,F}\simeq D_2*_{Z_2}D_m$ \cite[Theorem~2.1]{sahin_et_al06}.
\begin{figure}[h!]
\begin{tikzpicture}[scale=2.4]
\coordinate (z) at (-0.628341645367214,0);
\coordinate (m0) at (-0.589766908140269, -0.019510958298595112);
\coordinate (v0) at (0,-1);
\coordinate (t0) at (0,0);
\coordinate (m1) at (-0.589766908140269, 0.019510958298595088);
\coordinate (v1) at (0,1);
\coordinate (t1) at (-0.497293447571041, 0.4480459144733718);
\coordinate (m2) at (-0.6207560323954322, 0.04122881076904498);
\coordinate (v2) at (-0.8485737648598856, 0.5290770885150096);
\coordinate (t2) at (-0.804635184302833, 0.2784124121832999);
\coordinate (m3) at (-0.6545835002281661, 0.030776499177525736);
\coordinate (v3) at (-0.9761252742710891, 0.21720830768916605);
\coordinate (t3) at (-0.8922187314756402, 0.0883574310566999);
\coordinate (m4) at (-0.6681786379192989, 0);
\coordinate (v4) at (-1, 0);
\coordinate (t4) at (-0.8922187314756398, -0.08835743105670055);
\coordinate (m5) at (-0.6545835002281657, -0.030776499177525698);
\coordinate (v5) at (-0.976125274271089, -0.2172083076891669);
\coordinate (t5) at (-0.8046351843028328, -0.2784124121833005);
\coordinate (m6) at (-0.6207560323954323, -0.04122881076904487);
\coordinate (v6) at (-0.8485737648598849, -0.5290770885150109);
\coordinate (t6) at (-0.4972934475710395, -0.4480459144733719);

\draw (0,0) circle [radius=1cm];
\draw[line width=0.4pt] (v0) to (v1);
\draw[line width=0.4pt] (v1) to[arc through cw=(t1)] (v2);
\draw[line width=0.4pt] (v2) to[arc through cw=(t2)] (v3);
\draw[line width=0.4pt] (v3) to[arc through cw=(t3)] (v4);
\draw[line width=0.4pt] (v4) to[arc through cw=(t4)] (v5);
\draw[line width=0.4pt] (v5) to[arc through cw=(t5)] (v6);
\draw[line width=0.4pt] (v6) to[arc through cw=(t6)] (v0);

\draw[line width=0.4pt] (z) to[arc through cw=(m0)] (v0);
\draw[line width=0.4pt] (z) to[arc through ccw=(m1)] (v1);
\draw[line width=0.4pt] (z) to[arc through ccw=(m2)] (v2);
\draw[line width=0.4pt] (z) to[arc through ccw=(m3)] (v3);
\draw[line width=0.4pt] (z) to (v4);
\draw[line width=0.4pt] (z) to[arc through cw=(m5)] (v5);
\draw[line width=0.4pt] (z) to[arc through cw=(m6)] (v6);

\node[below] at (v0) {$0$};
\node[above] at (v1) {$\infty$};
\node[above left] at (v2) {$-\lambda$};
\node[left] at (v4) {$-1$};
\node[below left] at (v6) {$-\lambda\m$};
\node[right] at (t0) {$i$};
\node[below right] at (-0.64,-0.06) {$\zeta$};

\draw[dashed] (z) to (t0);
\draw[dashed] (t0) to (t2);
\draw[dashed] (t2) to[arc through cw=(v6)] (z);

\end{tikzpicture}
\caption{Fundamental domain for $\Gamma_7$ and its $R$-orbit (the dashed triangle refers to the proof of Theorem~\ref{ref21})}
\label{fig1}
\end{figure}

The action of $\PSL^\pm_2\Rbb$ extends naturally to the ideal boundary $\partial\mc{H}\simeq\PP^1\Rbb$ of $\mc{H}$. We write intervals in $\partial\mc{H}$ using the ``counterclockwise'' order induced by the disk model; thus $[0,-1]$ includes $\infty$ and excludes $-1/2$, while $[-1,0]$ is the closure of $\partial\mc{H}\setminus[0,-1]$. 
For $j\in\set{1,\ldots,m-1}$, the action of $A_j$ on the interval $[0,\infty]$ amounts to first mapping it inside $[\infty,0]$ by $R^{-j}$ (thus obtaining the $j$th, starting from $\infty$, of the intervals displayed to the left in Figure~\ref{fig1}), and then bringing it back inside $[0,\infty]$ by $S$. It is apparent that the union of the various $I_j=A_j[0,\infty]$ is all of $[0,\infty]$ and that, for $j\not=h$, $I_j$ and $I_h$ intersect (in a common endpoint) if and only if $j$ and $h$ are consecutive.

\begin{definition}
We\label{ref20} let $\Sigma$ be the monoid generated by the alphabet $\set{1,\ldots,m-1,f}$ modulo the relations $fj=(m-j)f$; also,
given any alphabet $\mc{A}$, we denote by $\mc{A}^*$ the free monoid of words over $\mc{A}$.
Clearly $\set{1,\ldots,m-1}^*$ is a submonoid of~$\Sigma$, and can be seen as the set of nodes of the full $(m-1)$-ary tree. We denote these nodes by bold letters $\bm{j}=j_0\cdots j_{l-1}$; thus $l$ is the \newword{level} of $\mb{j}$ in the tree, and
$\bm{h}$ is a descendant of $\bm{j}$ if and only if $\bm{j}$ is an initial segment of $\bm{h}$.
Since every element $\bm{s}$ of $\Sigma$ can be uniquely written either as $\bm{s}=\bm{j}$ or as $\bm{s}=\bm{j}f$, for some $\bm{j}\in\set{1,\ldots,m-1}^*$, we see that $\Sigma$ can be seen as the disjoint union of two copies of the full $(m-1)$-ary tree. 
We regard $A_{\argomento}$, as defined by~\eqref{eq2}, as a map from $\set{1,\ldots,m-1,f}$ to $\Gamma^\pm$, and define three more maps $B_{\argomento}$, $C_{\argomento}$, $\sharp$, all of domain $\set{1,\ldots,m-1,f}$, as follows:
\begin{gather*}
B_s = L A_s L\m,\\
C_j=\begin{pmatrix}
1/(m-1) & (j-1)/(m-1)\\
 & 1
 \end{pmatrix},\qquad
C_f=\begin{pmatrix}
-1 & 1\\
 & 1
 \end{pmatrix},\\
\sharp j=m-j,\qquad \sharp f=f.
\end{gather*}
\end{definition}

\begin{lemma}
The\label{ref4} maps $A_{\argomento}$ and $B_{\argomento}$ extend to 
isomorphic embeddings of $\Sigma$ into~$\Gamma^\pm$.
The map $C_{\argomento}$ extends to an isomorphic embedding of $\Sigma$
into the affine group $\Aff\Rbb$ of $\Rbb$, such that all maps in the range have coefficients
in the ring $\Zbb[1/(m-1)]$. The map $\sharp$ extends to an antiisomorphic involution of $\Sigma$ onto itself. We have $A_{\sharp\mb{s}}=A_{\mb{s}}^T$, the transpose of $A_{\mb{s}}$. 
\end{lemma}
\begin{proof}
Consider first $A_{\argomento}$; the above description of the intervals $I_1,\ldots,I_{m-1}$, together with the Ping-Pong Lemma~\cite[VII.A.2]{delaharpe00}, shows immediately that the monoid generated by $A_1,\ldots,A_{m-1}$ is free. Since the identity
$A_f A_j=A_f SR^{-j}=SA_f R^{-j}A_f A_f=SR^jA_f=A_{m-j}A_f$ holds, $A_{\argomento}$ extends naturally to a monoid homomorphism into $\Gamma^\pm$. 
As $A_f$ fixes $[0,\infty]$ we have $A_{\bm{j}f}[0,\infty]=A_{\bm{j}}[0,\infty]$, which readily implies that $A_{\argomento}$ is injective. This proves the statement about $A_{\argomento}$, and the one about $B_{\argomento}$ follows, since $B_{\argomento}$ is the composition of $A_{\argomento}$ with a conjugation.

Concerning $C_{\argomento}$, we are canonically identifying the affine group of $\Rbb$ with the subgroup of $\GL_2\Rbb$ whose elements have second row $\rppvector{0}{1}$. In that subgroup the identity $C_f C_j=C_{m-j} C_f$ holds, 
and thus $C_{\argomento}$ extends to a homeomorphism. The images of $[0,1]$ under $C_1,\ldots,C_{m-1}$ are the $m-1$ subintervals
of~$[0,1]$ of endpoints $k/(m-1)$, for $0\le k\le m-1$, and the previous argument for the injectivity of $A_{\argomento}$ applies here as well.

Finally, we extend $\sharp$ by setting $\sharp(s_0\cdots s_{l-1})=(\sharp s_{l-1})\cdots (\sharp s_0)$ for $s_i\in\set{1,\ldots,m-1,f}$. This definition respects the relation $fj=(m-j)f$, and thus defines an antiisomorphic involution of $\Sigma$. Both maps $A_{\bm{s}}\to A_{\sharp\bm{s}}$ and $A_{\bm{s}}\to A_{\bm{s}}^T$ are then antiisomorphic involutions on the range $A_\Sigma$ of $A_{\argomento}$, and it is readily seen that they agree on the generators $A_1,\ldots,A_{m-1},A_f$; thus they agree everywhere.
\end{proof}

\begin{remark}
The\label{ref18} isomorphism $C_{\argomento}\circ (B_{\argomento})\m$ from the submonoid $B_\Sigma$ of $\Gamma^\pm$ to the submonoid $C_\Sigma$ of $\Aff\Rbb$ does not extend to the enveloping groups:
$B_1\m B_2=L R\m L\m$ has order~$m$, but $C_1\m C_2=\ppmatrix{1}{1}{}{1}$ has infinite order. In~\S\ref{ref11} we will introduce a family of Minkowski functions ---one for each $m$--- that will upgrade this \emph{algebraic} isomorphism of monoids into a \emph{topological} conjugacy of actions.
\end{remark}

\section{Slow continued fractions}\label{ref5}

We recall that an \newword{iterated function system} (IFS for short) is a finite set $\ms{D}=\set{D_1,\ldots,D_q}$ of weakly contractive (i.e., $d(D_ix,D_iy)\le d(x,y)$ for every $i$) injective maps on a metric space $(X,d)$. The associated \newword{Hutchinson operator} is then the map $\Phi_{\ms{D}}$ on the power set of $X$ defined by
\[
\Phi_{\ms{D}}(W)=D_1[W]\cup\cdots\cup D_q[W];
\]
see~\cite{hutchinson81}, \cite{falconer14} for basics about IFSs.

\begin{definition}
A\label{ref6} \newword{decorated tree} is a finite subtree, not reduced to the root only, of the full $(m-1)$ary tree $\set{1,\ldots,m-1}^*$ such that any node which is not a leaf has precisely $m-1$ children and, moreover, some ---possibly none--- of the leaves $\bm{j}$ have been replaced by $\bm{j}f$.
Let $\bm{s}_1,\ldots,\bm{s}_q\in\Sigma$ be the leaves of a decorated tree,
$\ms{B}=\set{B_{\bm{s}_1},\ldots,B_{\bm{s}_q}}$, and $U$ the open interval $(0,\lambda_m\m)$. Clearly $\Phi_{\ms{B}}(U)$ is the disjoint union of $q$ open intervals, none of them depending on the decoration, covering $U$ up to finitely many points.
The \newword{Farey-type map} determined by $\set{\bm{s}_1,\ldots,\bm{s}_q}$ is then the piecewise-projective map $F:U\to U$ defined by $B_{\mb{s}_i}\m$ on $B_{\mb{s}_i}[U]$, and undefined on $U\setminus\Phi_{\ms{B}}(U)$.
\end{definition}

Decorated trees should be thought of as fragments of the Cayley graph of $\Gamma^\pm$, a point of view that will be relevant in~\S\ref{ref12}.

Letting $U'$ be the $G_\delta$ set $\bigcap_{n\ge1}\Phi_{\ms{B}}^n(U)$, it is easily established that $U'$ has full Lebesgue measure inside
$[0,\lambda\m]$, that all iterates of $F$ are defined on $U'$, and that $F\restriction U'$ is a Markov map with full branches.
Given $x\in U'$, the infinite sequence of indices of intervals
$B_{\bm{s}_1}[U],\ldots,B_{\bm{s}_q}[U]$ into which $x$ lands under iteration of $F$ can be seen as a continued fraction expansion of $x$; see \cite[Chapter~7]{CornfeldFomSi82}, \cite[Chapter~3]{einsiedlerward} and references therein for this interpretation of continued fractions in terms of dynamical systems. We feel then justified in looking at Farey-type maps as continued fraction algorithms; due to the presence of parabolic cycles these algorithms are called \newword{slow} or \newword{additive}.
Usually such maps are extended to all of $[0,\lambda\m]$ via some convention on the behaviour at endpoints; see Remark~\ref{ref13}(1).
Also, it is sometimes preferable to have maps defined on $[0,\infty]$; this is achieved by conjugating $F$ by $L\m$, which simply amounts to replacing every~$B$ matrix with the corresponding $A$ matrix.

\begin{example}
Of course $\Gamma_3^\pm=\PSL_2^\pm\Zbb$ is the extended modular group, and $\Gamma_3$ is the modular group; the Farey-type maps of Definition~\ref{ref6} are then precisely the slow continued fraction algorithms in~\cite[Definition~2.1]{panti18}. Consider the decorated trees in Figure~\ref{fig2}, in which black leaves correspond to undecorated nodes $\mb{j}$ and white leaves to decorated ones $\mb{j}f$.
\begin{figure}
\begin{tikzpicture}[level distance=2em, sibling distance=3em]
\coordinate
	child {[fill] circle (2pt)}
	child {circle (2pt)};
\end{tikzpicture}\qquad
\begin{tikzpicture}[level distance=2em, sibling distance=3em]
\coordinate
child {
	child {[fill] circle (2pt)}
	child {circle (2pt)}
	}
child {[fill] circle (2pt)};
\end{tikzpicture}\qquad
\begin{tikzpicture}[level distance=2em, sibling distance=3em]
\coordinate
child {
	child {[fill] circle (2pt)}
	child {[fill] circle (2pt)}
	}
child {[fill] circle (2pt)};
\end{tikzpicture}\qquad
\begin{tikzpicture}[level distance=2em, sibling distance=2.5em]
\coordinate
child {[fill] circle (2pt)}
child {[sibling distance=1.5em]
	child {circle (2pt)}
	child {circle (2pt)}
	child {[fill] circle (2pt)}
	}
child {circle (2pt)};
\end{tikzpicture}
\caption{Three decorated trees for $m=3$, and one for $m=4$}
\label{fig2}
\end{figure}
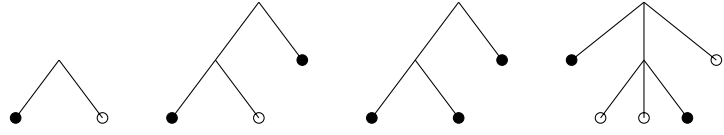
The first tree determines the classical Farey map, whose tent-like graph is well known~\cite[\S2]{bonanno_et_al13}. The maps determined by the other trees appear in Figure~\ref{fig3}, the first one being the Romik map~\cite{romik08}, also known as the even fractions map~\cite{bocalinden18}, \cite{cha_et_al18}; all pictures and computations in this paper are done using SageMath.
\begin{figure}[h!]
\includegraphics[width=3.5cm]{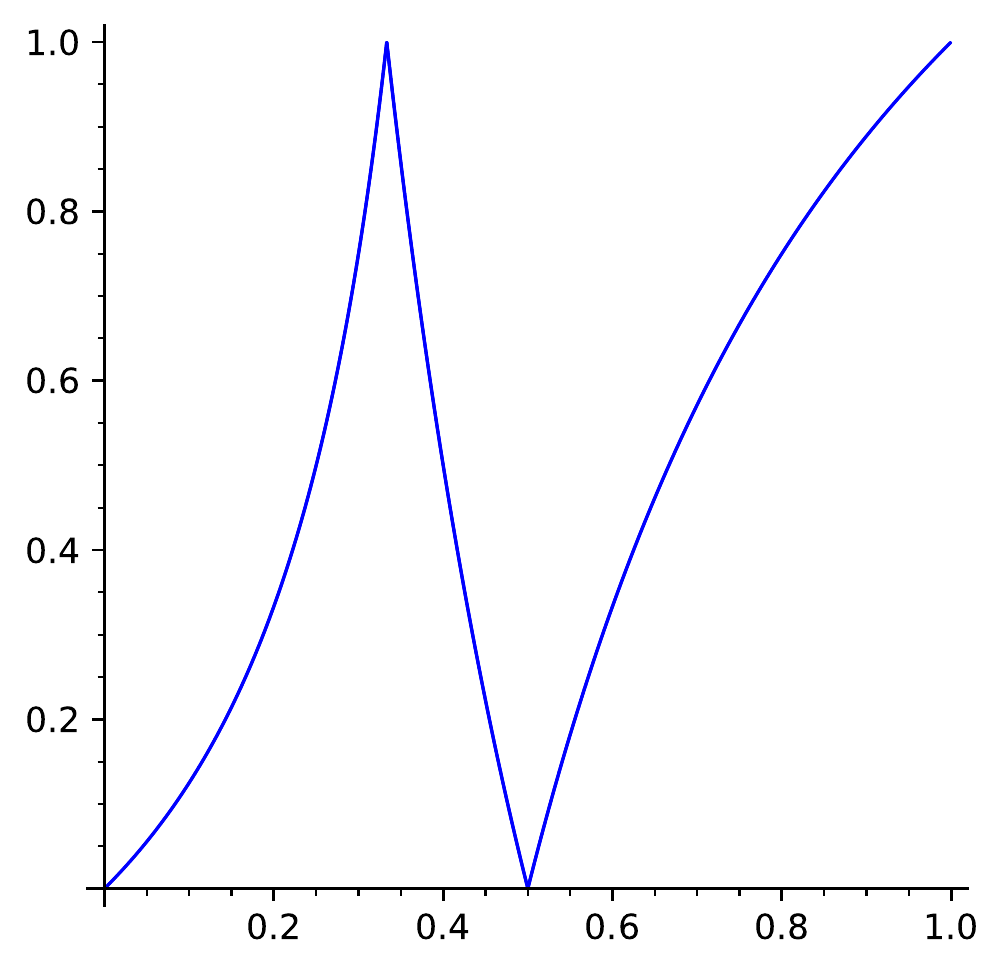}
\hspace{0.6cm}
\includegraphics[width=3.5cm]{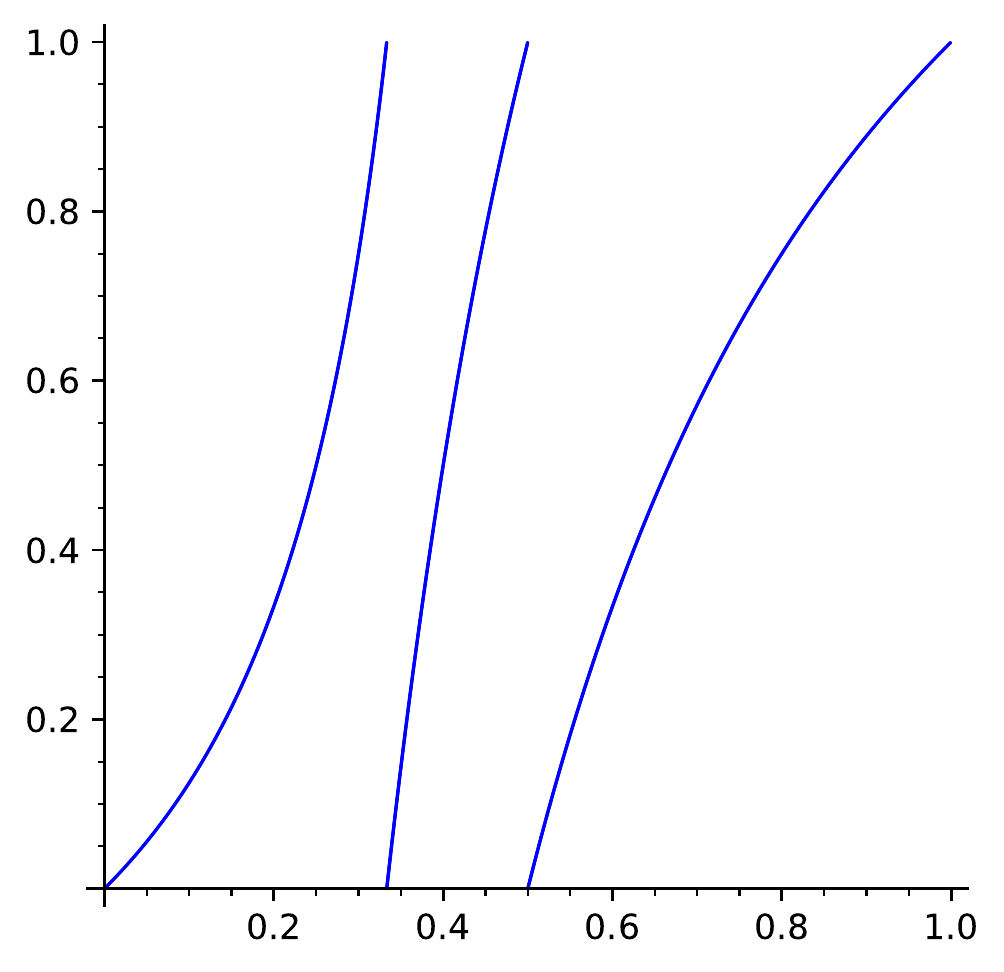}
\hspace{0.6cm}
\includegraphics[width=3.5cm]{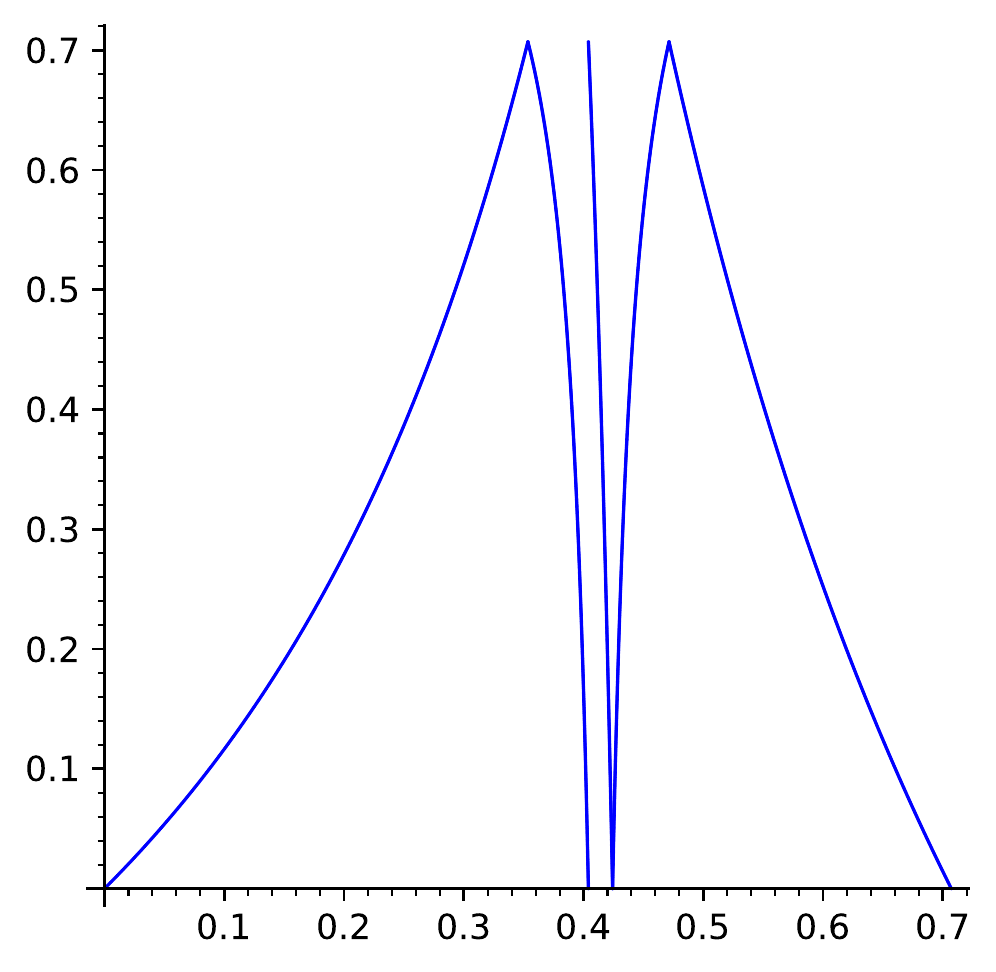}
\caption{The maps determined by the rightmost trees in Figure~\ref{fig2}}
\label{fig3}
\end{figure}
\end{example}

\begin{example}
In\label{ref7} Figure~\ref{fig4} left we plot the square of the Farey map; it is determined by the leaves $11,12f,21,22f$, with $m=3$. We plot to the right the map determined by $1,2f,3,4f$, for $m=5$, of domain $[0,\lambda_5\m]$, with $\lambda_5$ the golden ratio. These two maps are topologically conjugate, but their duals are not;
see Example~\ref{ref15}.
\begin{figure}[h!]
\includegraphics[width=4.5cm]{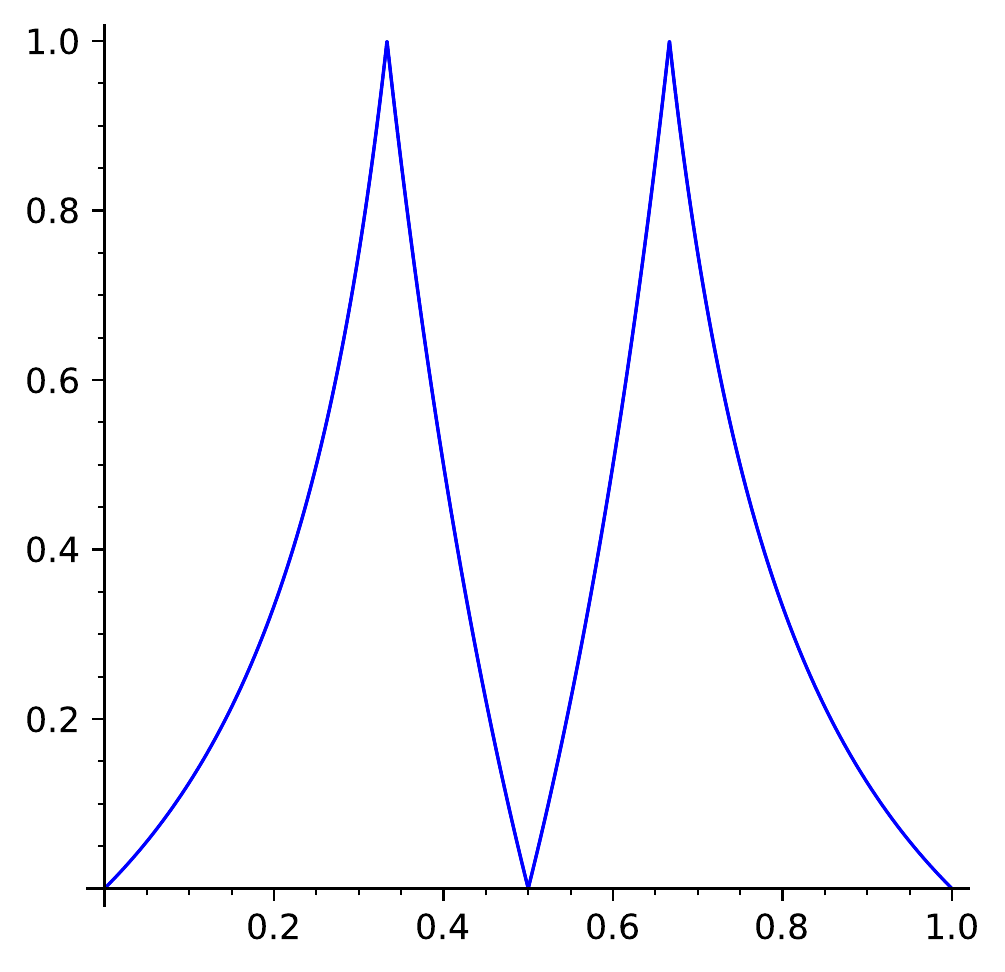}
\hspace{0.6cm}
\includegraphics[width=4.5cm]{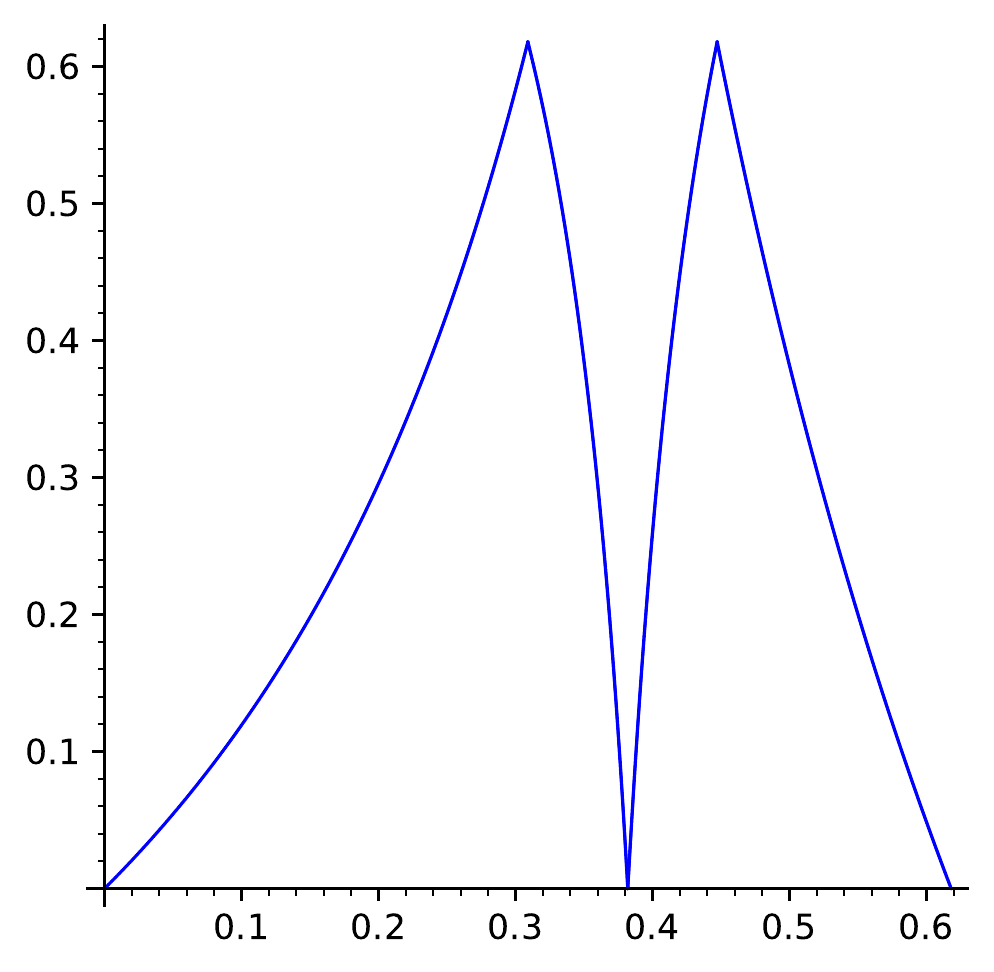}
\caption{Two topologically conjugate maps with non-topologically conjugate duals}
\label{fig4}
\end{figure}
\end{example}

We recall the basics of the theory of natural extensions. Given a measure-preserving dynamical system (a \newword{metric system} for short) $(X,\mu,F)$, with $\mu$ a Borel $F$-invariant measure, $\sigma$-finite but not necessarily finite, its \newword{natural extension} 
is an initial object $(X_e,\mu_e,F_e)$ in the category of invertible metric systems that have $(X,\mu,F)$ as a factor. In plain words, it is required that there exists a factor map $\varphi:X_e\to X$ intertwining $F_e$ with $F$, and such that every factor map from an invertible metric system to $(X,\mu,F)$ factors uniquely through~$\varphi$. Uniqueness of such an object follows from abstract nonsense, and its existence from mild requirements on the relevant categories (namely, closure under products and subsystems).

Given a Farey-type map $F$ as in Definition~\ref{ref6}, consider the IFS $\sharp\ms{B}=
\set{B_{\sharp\bm{s}_1},\ldots,B_{\sharp\bm{s}_q}}$, and let $K$ be its \newword{attractor}, namely the only closed nonempty subset of $[0,\lambda\m]$ that is a fixed point for $\Phi_{\sharp\ms{B}}$.
Equivalently, $K$ is the range of the continuous map $\pi_{\sharp\ms{B}}:\set{1,\ldots,q}^{\Zbb\p}\to[0,\lambda\m]$ implicitly defined by
\begin{equation}\label{eq6}
\bigcap_{n\ge0} B_{\sharp\bm{s}_{\omega(0)}}
B_{\sharp\bm{s}_{\omega(1)}}\cdots
B_{\sharp\bm{s}_{\omega(n-1)}}[0,\lambda\m]=\set{\pi_{\sharp\ms{B}}(\omega)},
\end{equation}
where $\omega=\omega(0)\omega(1)\cdots\in\set{1,\ldots,q}^{\Zbb\p}$.
Note that the intersection in~\eqref{eq6} is indeed a singleton, even though the IFS $\sharp\ms{B}$ is not strictly contracting; this follows from~\cite[Theorem~5.1]{panti20a}.

\begin{theorem}
Let\label{ref8} $U,U',F$ be as in Definition~\ref{ref6}, and $K$ as above.
\begin{enumerate}
\item The iterated function system $\sharp\ms{B}$ satisfies the strong open set condition. More precisely, there exists an open set $V\subset[0,\lambda\m]$ such that:
\begin{itemize}
\item[(1.1)] the closure $\overline{V}$ of $V$ equals $K$;
\item[(1.2)] $B_{\sharp\bm{s}_1}[V],\ldots,B_{\sharp\bm{s}_q}[V]$ are pairwise disjoint nonempty open subsets of $V$;
\item[(1.3)] $\Phi_{\sharp\ms{B}}(V)$ has full Lebesgue measure inside $V$.
\end{itemize}
\item Let $F_\sharp:V\to V$ be defined by $B_{\sharp\mb{s}_i}\m$ on $B_{\sharp\mb{s}_i}[V]$ and undefined otherwise; let also
$V'=\bigcap_{n\ge1}\Phi_{\sharp\ms{B}}^n(V)$. Then $V'$ has full Lebesgue measure inside $V$, the closure of $V'$ equals $K$, and
$F_\sharp \restriction V'$ is a Markov map with full branches.
\item Let $\eta$ be the infinite $\sigma$-finite measure determined by
\begin{equation}\label{eq3}
\ud\eta=\frac{\ud x\ud y}{\bigl(1-\lambda x-\lambda y+(\lambda^2+1)xy\bigr)^2}
\end{equation}
on $U\times V$. Then the two maps
\begin{equation}\label{eq4}
\begin{split}
F_e(x,y)&=\bigl(F(x),B_{\sharp\bm{s}_{i(x)}}(y)\bigr),\\
F_e\m(x,y)&=\bigl(B_{\bm{s}_{i(y)}}(x),F_\sharp(y)\bigr),
\end{split}
\end{equation}
(the indices $i(x)$ and $i(y)$ being defined by $x\in B_{\mb{s}_{i(x)}}[U]$ and
$y\in B_{\sharp\mb{s}_{i(y)}}[V]$)
from $U\times V$ to itself preserve $\eta$ and are inverse to each other up to $\eta$-nullsets.
\item The pushforward measure $\mu=(\pi_1)_*\eta$ is the unique $F$-invariant measure on $U$ which is absolutely continuous w.r.t.~Lebesgue measure. The system $(U\times V,\eta,F_e)$ is the natural extension of $(U,\mu,F)$.
\item The measure
\begin{equation}\label{eq5}
\ud\nu=\frac{\ud y}{(1-\lambda y)y}
\end{equation}
on $V$ is $F_\sharp$-invariant, and $(U\times V,\eta,F_e\m)$ is the natural extension of $(V,\nu,F_\sharp)$.
\item The above statements remain true for the version of $F$ over $[0,\infty]$, by simply replacing every $B$ matrix with the corresponding $A$ matrix, the right side of~\eqref{eq3} with the ``standard number-theoretic measure'' $(1-xy)^{-2}\ud x\ud y$ (see, e.g.,~\cite[p.~718]{arnouxschmidt13}), and the right hand side of~\eqref{eq5} with $y\m\ud y$.
\end{enumerate}
\end{theorem}

The main statement of Theorem~\ref{ref8} is~(1), the other items following familiar lines. Our results should be compared with those in~\cite{arnouxschmidt19}, in which natural extensions for piecewise projective maps also result from the construction of an attractor. The approach of~\cite{arnouxschmidt19} is more general, as it covers multidimensional cases and Markov maps with not necessarily full branches. Of course this has a price: the attractor to be constructed is not $1$-dimensional as our $K$, but higher-dimensional and of elusive visualization. More relevantly, our simpler product structure yields that not only $F_e$ is a skew product over $F$, but $F_e\m$ is a skew product as well, over the new map $F_\sharp$. If the original map does not have full branches, then the composition of $F_e\m$ with projection on the second component may not be a function of the variable $y$ alone; in such cases $F_\sharp$ does not exist. A good example of this phenomenon is in~\cite{kraaikamplangeveld17}; see the jagged boundary between the regions $\mc{T}(\Delta_3)$ and $\mc{T}(\Delta_4)$
in~\cite[Fig.~16 right]{kraaikamplangeveld17}.

\begin{definition}
We\label{ref9} call the metric system $(V,\nu,F_\sharp)$ the \newword{dual} of $(U,\mu,F)$. If $F$ is conjugate to $F_\sharp$ via an element of $\Gamma^\pm$, then we say that $F$ is \newword{selfdual}.
\end{definition}

We will proceed as follows: we close this section by proving Theorem~\ref{ref8}(2)--(6) under the assumption of (1). 
In \S\ref{ref10} we will provide examples, some of which of independent interest. In \S\ref{ref11}, making no use of~(1), we will construct, for each $m=3,4,5,\ldots$, a homeomorphism $M_m:[0,\lambda_m\m]\to[0,1]$ that conjugates the action of $B_{\bm{s}}$ with that of $C_{\bm{s}}$, for every $\bm{s}\in\Sigma$; note that $M_3$ is the classical Minkowski question mark function~\cite{denjoy38}, \cite{salem43}, \cite{jordansahlsten16}. In Theorem~\ref{ref21} we will explicitly determine the H\"older exponent of each $M_m$.
Finally in~\S\ref{ref12}, relying on the results of~\S\ref{ref11}, we will prove Theorem~\ref{ref8}(1).

\begin{remark}\label{ref13}
\begin{enumerate}
\item It is expedient to think of $F$ and $F_\sharp$ to be defined on all of $[0,\lambda\m]$ and $K$, respectively. Using Theorem~\ref{ref8}(2) we simply agree to set $F_\sharp(y)=B_{\sharp\mb{s}_{i(y)}}\m(y)$, where $i(y)$ is the minimum index $i\in\set{1,\ldots,q}$ such that $y\in K$ is an accumulation point for $B_{\sharp\mb{s}_i}[V']$; analogous conventions apply to $F$.

\item As implicit in Definition~\ref{ref9}, we consider conjugate maps as the same map. Our definition of $F_\sharp$ relies on the identification of the complementary interval $[\lambda\m,0]$ with $[0,\lambda\m]$ via the involution $LSL\m$. We could have identified the two intervals via
$L\ppmatrix{-1}{}{}{1}L\m$, which is still in $\Gamma^\pm$, or even have avoided any identification. However, the resulting dual map would have been conjugate to our~$F_\sharp$, so these choices are irrelevant.

\item Besides the identity, the only element of $\Gamma^\pm$ that fixes $[0,\lambda\m]$ globally is~$B_f$. 
Conjugating $F$ by $B_f$ amounts to flipping the decorated tree determining $F$ around its vertical axis. This gives nothing new, as the resulting dual is the old one conjugated by $B_f$.

\item Given $F$, its $n$-fold composition $F^n$ is again a Farey-type map, and $(F^n)_\sharp=(F_\sharp)^n$. Indeed, let $F$ be determined by the set of decorated leaves $\mb{S}=\set{\mb{s}_1,\ldots,\mb{s}_q}\subset\Sigma$. Then $F^n$ is determined by the set of all $n$-long products of elements of $\mb{S}$.
Since the $\sharp$-image of such a product is the product of the $\sharp$-images of the factors in inverse order, $(F^n)_\sharp$ is determined by the set of $n$-long products of elements of $\sharp[\bm{S}]$. This set has as attractor the same attractor of the original $\sharp[\bm{S}]$, and the identity $(F^n)_\sharp=(F_\sharp)^n$ follows by a straightforward argument.
\end{enumerate}
\end{remark}

\begin{proof}[Proof of Theorem~\ref{ref8}(2)--(6) assuming (1)]
(2) It is easy to show that
\[
V\setminus V'\subseteq \bigcup_{n\ge0}\Phi^n\bigl(V\setminus\Phi(V)\bigr).
\]
All maps in $\sharp\ms{B}$ transform Lebesgue nullsets into Lebesgue nullsets; therefore, by~(1.3), $V\setminus V'$ is a Lebesgue nullset.
This implies that the open set $V\setminus\overline{V'}$ is empty and $V\subseteq\overline{V'}$; thus $K=\overline{V}\subseteq\overline{V'}\subseteq\overline{V}=K$.
It is also plain that $V'$ is the disjoint union of $B_{\sharp\mb{s}_1}[V'],\ldots,
B_{\sharp\mb{s}_q}[V']$, none of the latter sets being empty; thus $F_\sharp\restriction V'$ is a Markov map as claimed.

(3) Both $F_e$ and $F_e\m$ are defined on $U'\times V'$, and so are all iterates. From their very definition, they
are inverse to each other.
In order to prove the stated invariance of measures, we temporarily switch to the version on $[0,\infty]$. This is irrelevant, because the pullback of the conjugating diffeomorphism $L:[0,\infty]\to[0,\lambda\m]$ maps the forms in~\eqref{eq3} and~\eqref{eq5} to those specified in~(6).
Let $\mf{A}<\PSL_2\Rbb$ be the subgroup of diagonal matrices. Then $\PSL^\pm_2\Rbb$ acts to the left on the homogeneous space $\PSL_2\Rbb/\mf{A}$ via
\[
A(E\mf{A})=AE\begin{pmatrix}
\det(A) & \\
 & 1
\end{pmatrix}\mf{A},
\]
and this action preserves the quotient Haar measure (see, e.g.,~\cite[Theorem~4.1]{panti20b}). We can identify $\PSL_2\Rbb/\mf{A}$ with the space of oriented geodesics in the hyperbolic plane, in turn identified ---by looking at the geodesics' limit points--- with the open cylinder $\set{(\omega,\alpha):\omega,\alpha\in\PP^1\Rbb\text{ and }\omega\not=\alpha}$. The action is then componentwise $A(\omega,\alpha)=(A(\omega),A(\alpha))$, and the quotient Haar measure is to the one induced by the form $(\omega-\alpha)^{-2}\ud\omega\ud\alpha$.
Let $\bm{s}_i$ be one of the leaves determining our map~$F$; then $A_{\bm{s}_i}\m$ expands $I_{\bm{s}_i}$ to $[0,\infty]$, and contracts $[\infty,0]$ inside itself. The change of variables $(\omega,\alpha)=Q(x,y)=(x,S(y))$ conjugates the action of $A_{\bm{s}_i}\m$ on $I_{\bm{s}_i}\times[\infty,0]$ with that of $F_e$ on $I_{\bm{s}_i}\times[0,\infty]$;
moreover, the pullback of $(\omega-\alpha)^{-2}\ud\omega\ud\alpha$ by $Q$ is $(1-xy)^{-2}\ud x\ud y$. Since, as discussed above, $A_{\bm{s}_i}\m$ preserves the measure induced by the former, $F_e$ must preserve the measure induced by the latter.

We conclude the proof of~(4)--(6) by reverting to the version of $F$ on $[0,\lambda\m]$ and proving the uniqueness of $\mu$ stated in~(4), as well as the fact that $F_e$ and $F_e\m$ are the natural extensions of $F$ and $F_\sharp$, respectively.
By~\cite{thaler83}, transformations such as $F$ admit precisely one invariant measure absolutely continuous w.r.t.\ Lebesgue measure; clearly $\mu$ is such a measure, and this settles uniqueness.
In order to show that $F_e$ is the natural extension of $F$, we must show the following~\cite[p.~22]{rohlin61}:
\begin{itemize}
\item for $\eta$-almost all pairs $(x,y),(x',y')$ of distinct points, there exists $n\ge0$ such that $\pi_1\bigl(T_e^{-n}(x,y)\bigr)\not=\pi_1\bigl(T_e^{-n}(x',y')\bigr)$.
\end{itemize}

We may assume $(x,y),(x',y')\in U'\times V'$ with $x=x'$ (otherwise take $n=0$). Since $y\not=y'$ and
the intersection in~\eqref{eq6} is always a singleton, there must exist $t\ge0$ such that $i\bigl(F_\sharp^t(y)\bigr)\not= i\bigl(F_\sharp^t(y')\bigr)$, which in turn yields
$\pi_1\bigl(T_e^{-(t+1)}(x,y)\bigr)\not=\pi_1\bigl(T_e^{-(t+1)}(x,y')\bigr)$, as desired. The same arguments shows that $F_e\m$ is the natural extension of $F_\sharp$. This concludes the proof of Theorem~\ref{ref8}(2)--(6), under the assumption of~(1).
\end{proof}

\section{Examples: the number of components of $K$}\label{ref10}

The attractor $K$ of $\sharp\ms{B}$ is always a perfect set (i.e., closed without isolated points), but never a Cantor set.
This follows from Theorem~\ref{ref8}(1), that we are still assuming in this section; it will be proved in the next two sections, which do not depend on the present one.
We classify $K$ according to the number of its connected components; each component is either a closed interval of positive length, or a single point.

\begin{proposition}
Let\label{ref16} the attractor $K$ be as above; then precisely one of the following statements holds.
\begin{itemize}
\item[(i)] $K$ has finitely many components, and all of them are intervals;
\item[(ii)] $K$ has countably many interval components and at most countably many point components;
\item[(iii)] $K$ has countably many interval components and $2^{\aleph_0}$ point components.
\end{itemize}
\end{proposition}
\begin{proof}
Let $x=\pi_{\sharp\ms{B}}(\omega)\in K$ for some $\omega\in\set{1,\ldots,q}^{\Zbb\p}$; then
\[
\set{x}=\bigcap_{n\ge0} B_{\sharp\bm{s}_{\omega(0)}}\cdots
B_{\sharp\bm{s}_{\omega(n-1)}}[K].
\]
If $x$ is a single-point component then, since by Theorem~\ref{ref8}(1) $K$ contains at least one interval, and $x$ does not belong to any interval component, $x$ is an accumulation point of interval components. In particular, the existence of a point-component implies the existence of countably many interval-components. The three statements above are of course mutually exclusive. Suppose both (i) and (ii) are false; then there are countably many interval components and uncountably many point components (we are not assuming the continuum hypothesis). By an easily definable continuous function we can shrink every interval component to a point. The resulting image set is nonempty and still perfect, and thus has cardinality $2^{\aleph_0}$~\cite[Lemma~4.2]{jech78}; this establishes~(iii).
\end{proof}

\begin{example}
For every $m$ the map determined by the set of leaves $\set{1,2,\ldots,m-1}$ is selfdual; indeed, that set is invariant under the involution~$\sharp$ of Lemma~\ref{ref4}. These are the maps used in~\cite[p.~2185]{pohl14}, Pohl's $g_k$ being our $A_{m-k}$; their invariant density is $x\m\ud x$ in the $[0,\infty]$ version, and $x\m(1-\lambda x)\m\ud x$ in the $[0,\lambda\m]$ version. The same statements hold for the set of leaves $\set{1f,2f,\ldots,(m-1)f}$ and, whenever $m$ is even, for the saw-like maps $\set{1,2f,3,4f,\ldots,(m-2)f,m-1}$ and $\set{1f,2,3f,4,\ldots,m-2,(m-1)f}$.
\end{example}

\begin{example}
Let\label{ref15} us compute the duals of the maps in Figure~\ref{fig4}. The one to the left is the square of the Farey map $F$; by Remark~\ref{ref13}(4), we may simply compute the dual of $F$ and take its square. The leaves determining $F$ are $1,2f$, whose $\sharp$-images are $2,2f$. We have then to compute the attractor of
$\sharp\ms{B}=\set{B_2,B_{2f}}$, which is easy. Indeed, the image of $[0,1]$ under both $B_2=\ppmatrix{}{1}{-1}{2}$ and $B_{2f}=\ppmatrix{}{1}{1}{1}$
is $[1/2,1]$; this latter interval is the attractor, since it is the union of its $B_2$-image $[2/3,1]$ and its $B_{2f}$-image $[1/2,2/3]$. Thus $F_\sharp$ is induced by $B_{2f}\m$ on $[1/2,2/3]$, and by $B_2\m$ on $[2/3,1]$; we must now consider $F_\sharp^2$, which is shown in Figure~\ref{fig5} left.
It is easily seen that $F$ and $F_\sharp$ are conjugate by $\ppmatrix{}{1}{1}{1}$, so all powers of the Farey map are selfdual.
\begin{figure}[h!]
\includegraphics[width=4.5cm]{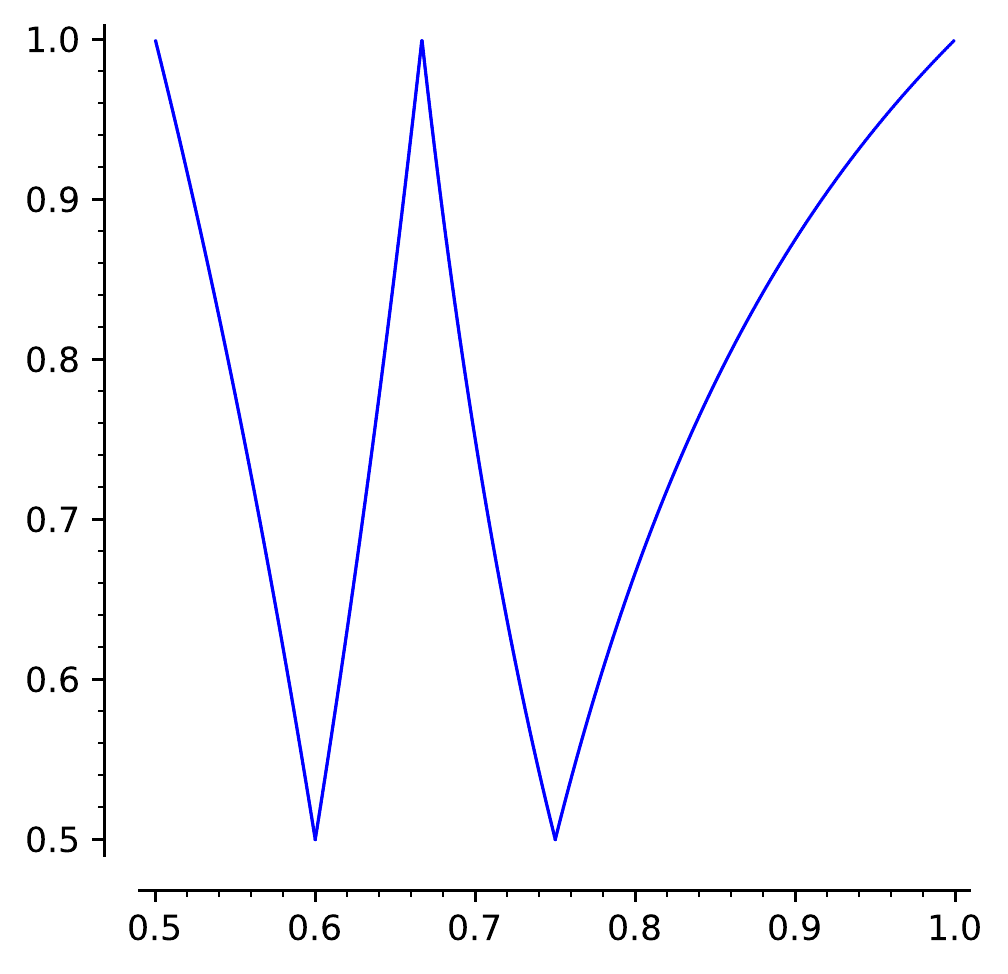}
\hspace{0.6cm}
\includegraphics[width=4.5cm]{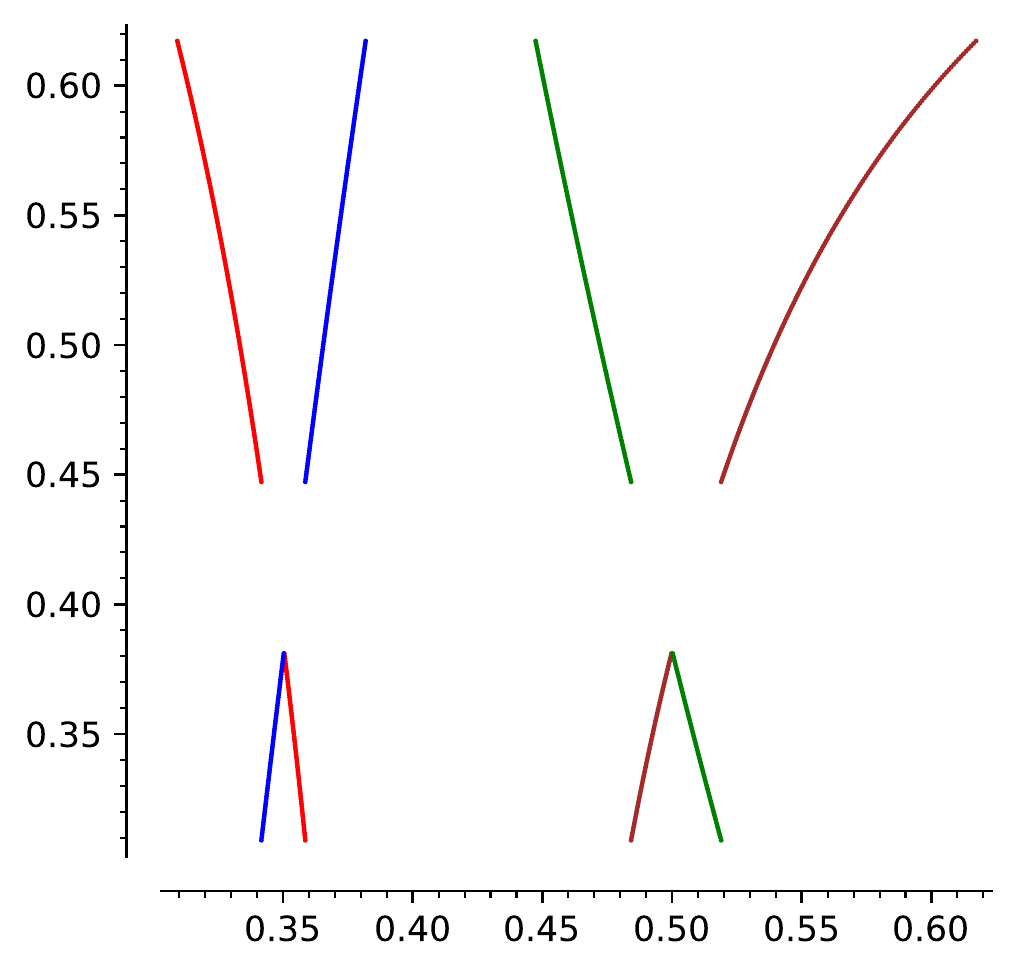}
\caption{The duals of the maps in Figure~\ref{ref4}}
\label{fig5}
\end{figure}

The map to the right of Figure~\ref{fig4} is determined by the leaves $1,2f,3,4f$; here $m=5$ and $\lambda$ is the golden ratio. The $\sharp$-images of those leaves are, respectively, $4,2f,2,4f$, and we have
\begin{align*}
B_4&=\begin{pmatrix}
-\lambda & \lambda\\
-1-2\lambda & 2+\lambda
\end{pmatrix},
&B_{2f}&=\begin{pmatrix}
-\lambda & \lambda\\
-2-\lambda & 1+2\lambda
\end{pmatrix},\\
B_2&=\begin{pmatrix}
 & 1\\
-1 & 2\lambda
\end{pmatrix},
&B_{4f}&=\begin{pmatrix}
 & 1\\
1 & \lambda
\end{pmatrix}.
\end{align*}
Now, $B_4[0,\lambda\m]=B_{4f}[0,\lambda\m]=\bigl[(-1+2\lambda)/5,\lambda\m\bigr]$
and
$B_{2f}[0,\lambda\m]=B_2[0,\lambda\m]=\bigl[(-1+\lambda)/2,2-\lambda\bigr]$.
The union of these two intervals is the attractor $K$; indeed, its $\Phi_{\sharp\ms{B}}$-image is the union of eight intervals, that glue together along endpoints to reconstruct $K$. We draw the graph of $F_\sharp$ in Figure~\ref{fig5} right, the branch $B_4\m$ being drawn in brown, $B_{2f}\m$ in red, $B_2\m$ in blue, and $B_{4f}\m$ in green. Clearly $F$ and $F_\sharp$ are not conjugate.
\end{example}

\begin{example}
Let\label{ref14} $m=3$. Remember that the Gauss map $G(x)=1/x-\floor{1/x}$ can be seen as Schweiger’s jump transformation~\cite[Chapter~18]{schweiger95} of the Farey map $F$ w.r.t.~the entrance time in~$[1/2,1]$; more precisely, $G(x)=F^{e(x)+1}(x)$, where
\[
e(x)=\min\set{t\ge0:F^t(x)\in[1/2,1]}.
\]
Let us replace, for $n=1,2,3,\ldots$, the entrance time with the $(n-1)$-truncated entrance time, namely
\[
e_n(x)=\min\set{t\ge0:F^t(x)\in[1/2,1]\text{ or }t\ge n-1}.
\]
We obtain $F_n(x)=F^{e_n(x)+1}(x)$, $F_1=F$, and $F_\infty=G$. 
Then $F_n$ is a slow map, determined by the set of leaves $\set{1^n}\cup\set{1^q2f:n>q\ge0}$, whose $\sharp$-image is $\set{2^n}\cup\set{21^qf:n>q\ge0}$.
We spare the reader the computation of the attractor, which turns out to be
\[
K_n=
\bigcup_{k\ge0}\biggl[\frac{kn+1}{kn+2},\frac{kn+2}{kn+3}\biggr]
\cup\set{1},
\]
hence of the type in Proposition~\ref{ref16}(ii).
In Figure~\ref{fig6} we draw from left to right the graph of $F_3$, the IFS $\sharp\ms{B}=\set{B_{222},B_{211f},B_{21f},B_{2f}}$, and the dual map ${F_3}_\sharp$.
\begin{figure}[h!]
\includegraphics[width=4.1cm]{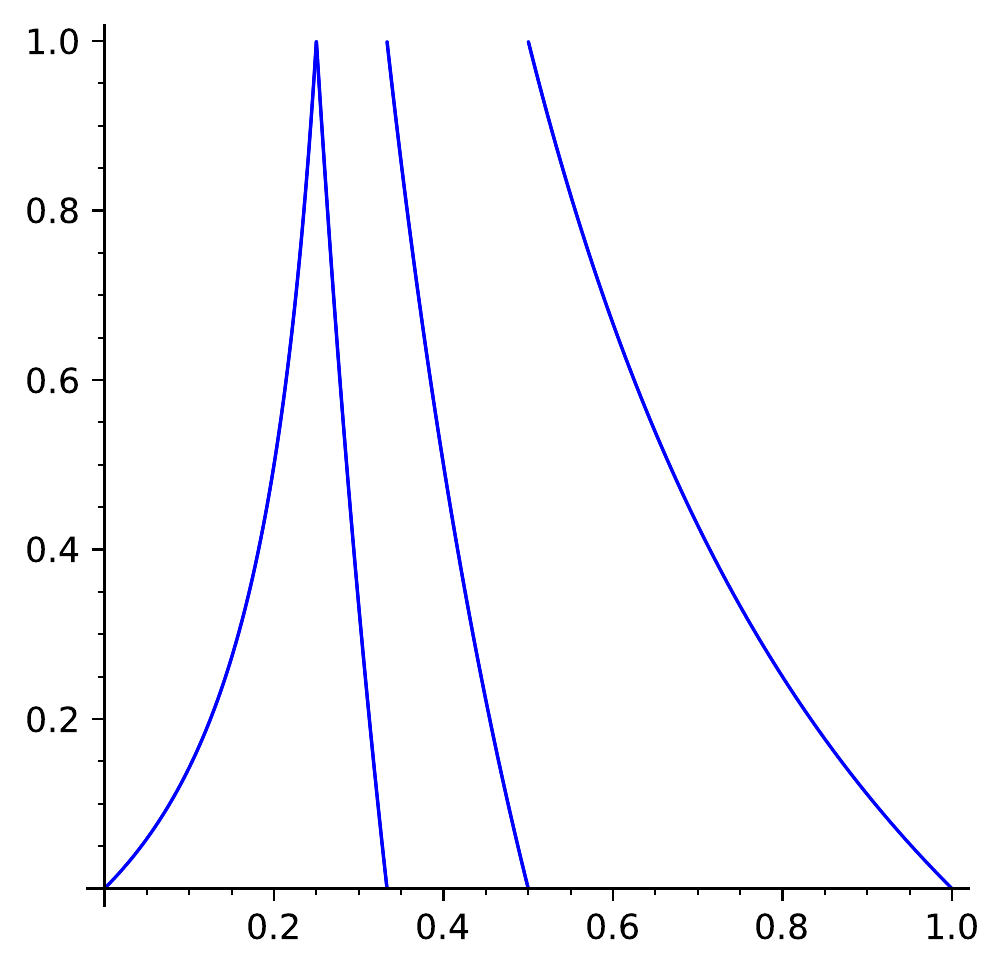}
\hspace{0.0cm}
\includegraphics[width=4.1cm]{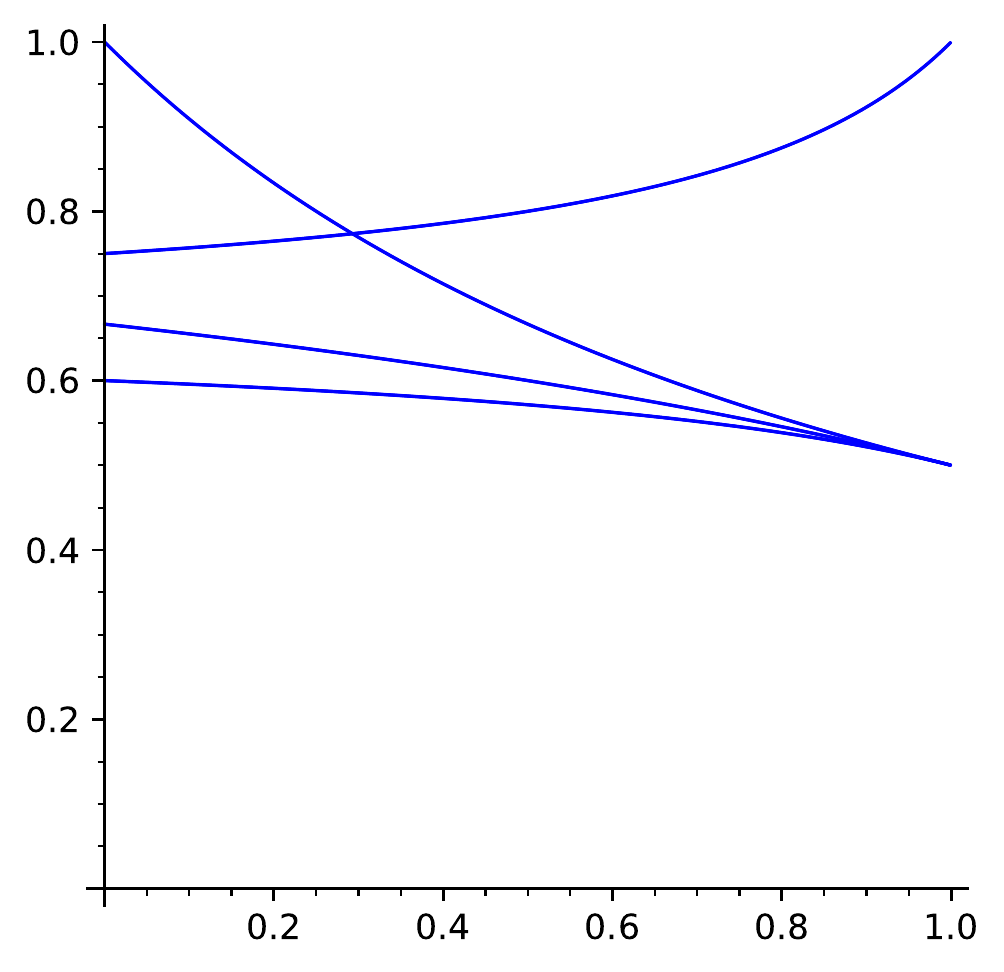}
\hspace{0.0cm}
\includegraphics[width=4.1cm]{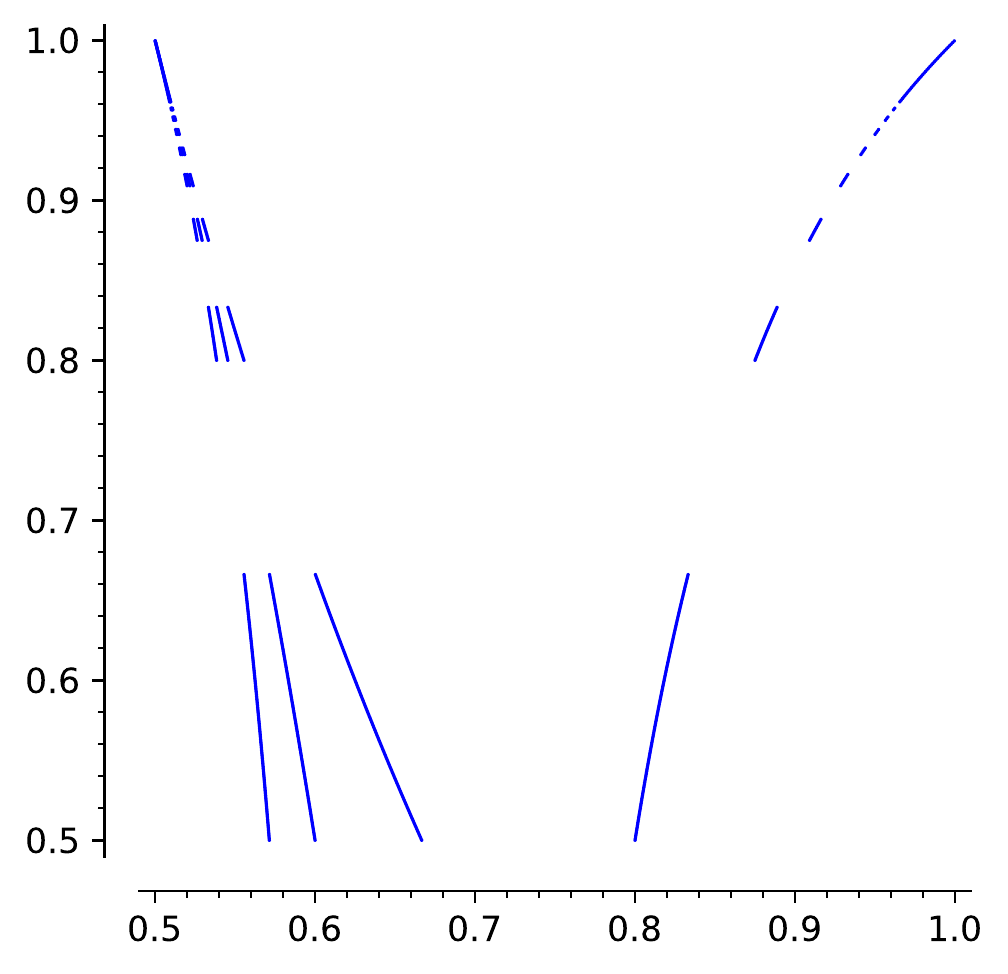}
\caption{The map $F_3$ of Example~\ref{ref14}, its dual IFS, and ${F_3}_\sharp$}
\label{fig6}
\end{figure}

Theorem~\ref{ref8}(4) allows the explicit computation of the $F_n$-invariant
density $h_n$, namely
\begin{align*}
h_n(x)&=\sum_{k\ge0}
\int_{\frac{kn+1}{kn+2}}^{\frac{kn+2}{kn+3}}
\frac{1}{(1-x-y+2xy)^2} \ud y\\
&=\sum_{k\ge0}\frac{1}{(knx+x+1)(knx+1)}\\
&=\frac{1}{nx^2}\biggl[\psi\biggl(\frac{1}{nx}+\frac{1}{n}\biggr)-
\psi\biggl(\frac{1}{nx}\biggr)\biggr],
\end{align*}
where $\psi$ is the digamma function. The functional equation $\psi(s+1)=\psi(s)+1/s$ yields $h_1(x)=1/x$, as expected. On the other hand, 
by the Dominated Convergence Theorem,
for every $0<x\le1$ we have $\lim_{n\to\infty}h_n(x)\to 1/(1+x)$ which is, up to the normalizing constant, the invariant density of $G$.
\end{example}

\begin{example}
At\label{ref17} the end of~\cite[\S7]{arnouxschmidt19} it is asked if it is true that in the orientation-preserving case the domain of the natural extension is always connected. We provide a negative answer by giving an example over the integers (thus $m=3$ again), with undecorated leaves and such that $K$ is of the type in Proposition~\ref{ref16}(iii). Let $F$ 
(see Figure~\ref{fig7} left)
be determined by the six leaves $1^3,1^22,12,21,2^21,2^3$. It is easy to see that the branches of $F$ generate $\PSL_2\Zbb$ (this is a special case of~\cite[Corollary~4.3]{panti18}).
\begin{figure}[h!]
\includegraphics[width=6.0cm]{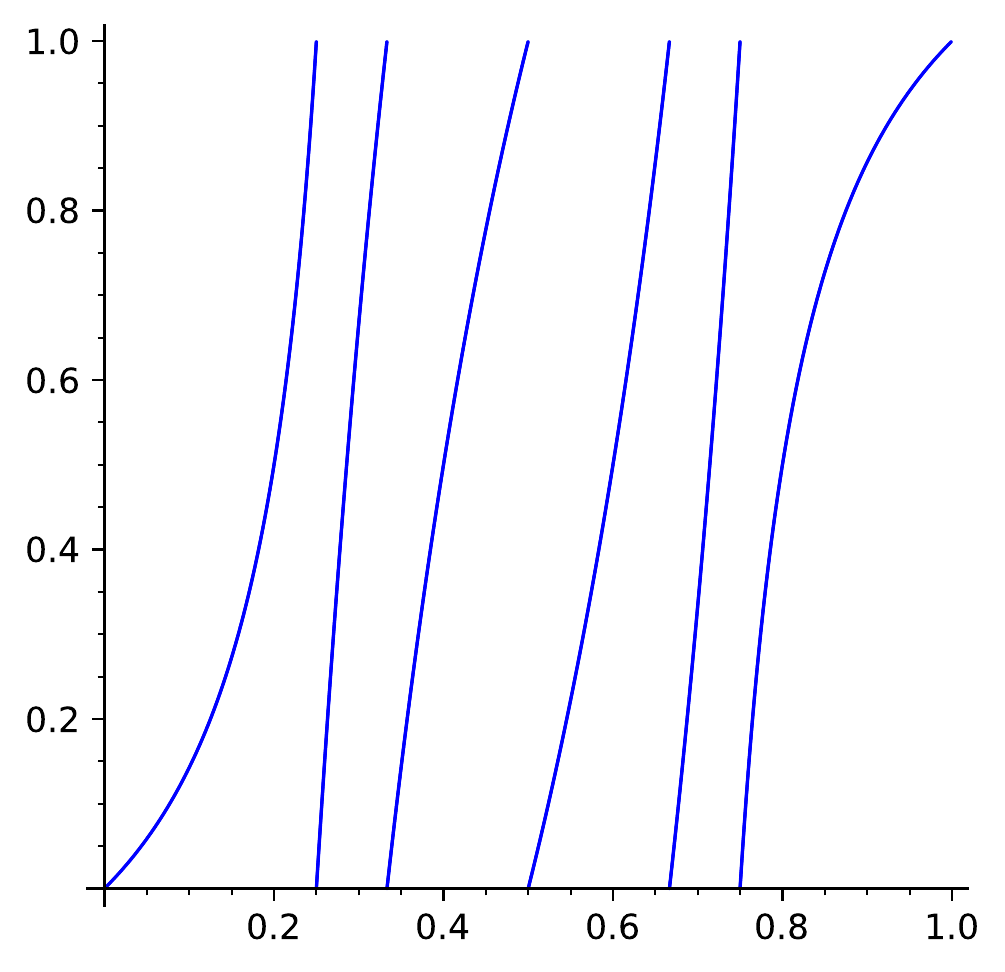}
\hspace{0.0cm}
\includegraphics[width=6.0cm]{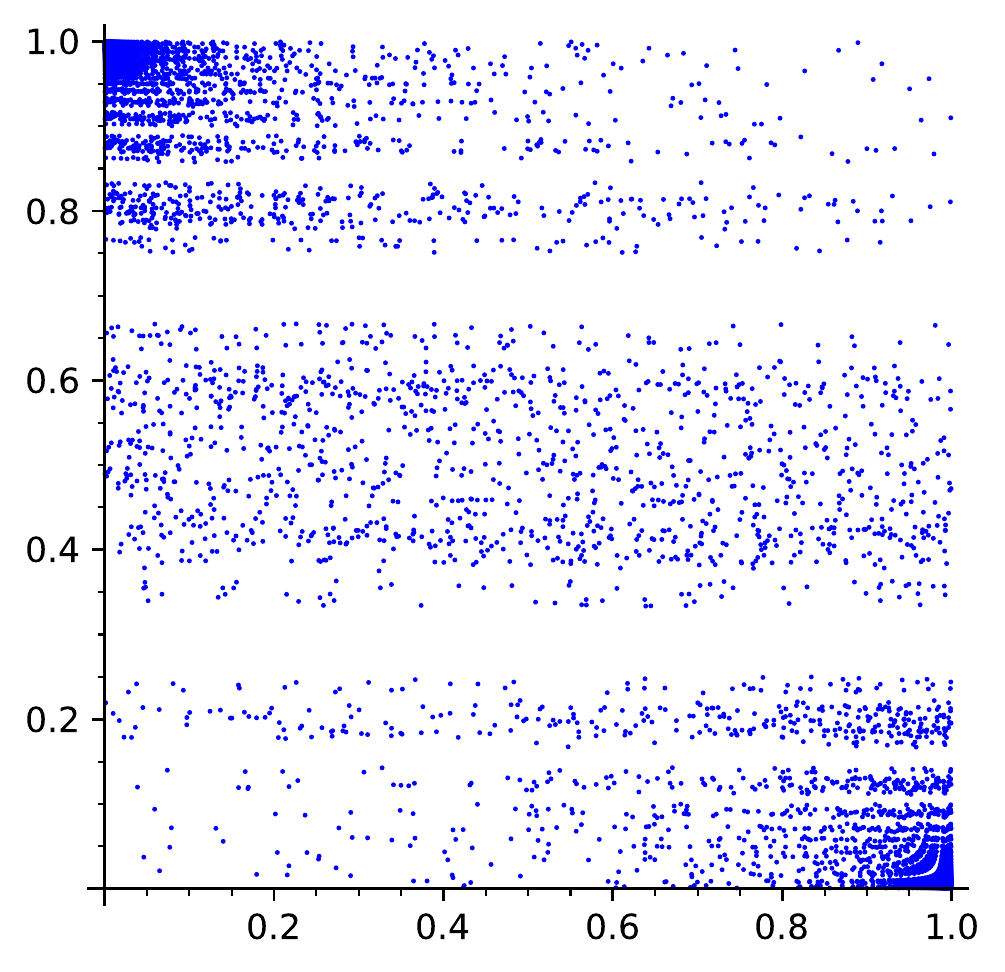}
\caption{The map $F$ of Example~\ref{ref17} and a generic $F_e$-orbit}
\label{fig7}
\end{figure}

For simplicity's sake we rename the elements of $\sharp\ms{B}=\set{B_{2^3},B_{12^2},B_{12},B_{21},B_{21^2},B_{1^3}}$ as $D_1,D_2,\ldots,D_6$, the numbering being determined by $D_1(1/2)<D_2(1/2)<\cdots<D_6(1/2)$; also,
for $w=j_0\cdots j_{n-1}\in\set{1,\ldots,6}^*$, we let $\overline{w}=D_{j_0}\cdots D_{j_{n-1}}[0,1]$ and $K_n=\bigcup_{\length(w)=n}\overline{w}$. We plan to show that $K$ has uncountably many point-components, and for this purpose it is enough to show that, for every $n$ and every word $v$ of length $n$ over the restricted alphabet $\set{1,6}$, the closed interval $\overline{v}$ is a component of $K_n$. This amounts to showing that for every word $w\not=v$, again of length~$n$ but over the full alphabet, the intervals $\overline{v}$ and $\overline{w}$ are disjoint, which is easily accomplished by induction on $n$. Indeed, if $v$ and $w$ begin with different digits the statement follows from the easily checked fact that $\overline{1}$ and $\overline{6}$ do not intersect each other, nor intersect any of $\overline{2},\ldots,\overline{5}$; in particular this covers the case $n=1$. Let $n\ge2$ and assume, say, $v=1v',w=1w'$. Then $v'\not=w'$ and $\overline{v}\cap\overline{w}=D_1[\overline{v'}]\cap D_1[\overline{w'}]=
D_1[\overline{v'}\cap\overline{w'}]=D_1[\emptyset]=\emptyset$, by inductive hypothesis.

In Figure~\ref{fig7} right we plot $8000$ points of the $F_e$-orbit of $(\lambda_7-1,1/2)$. Here $\lambda_7-1$ appears only as a ``generic'' cubic irrational, allowing the computation to be approximation-free. The fact that algebraic numbers of degree higher than $2$ should behave (at least empirically) in a generic way under c.f.~maps over the integers is still highly mysterious~\cite{bugeaud15}.
Of course, the density points of the orbit in the upper left and lower right corners correspond to the two parabolic fixed points.
\end{example}

\section{Minkowski functions}\label{ref11}

It is time to let the embedding $\mb{s}\mapsto C_{\mb{s}}$ of Definition~\ref{ref4} come into play. Given~$m$, let $\ms{B}=\set{B_1,\ldots, B_{m-1}}$
and $\ms{C}=\set{C_1,\ldots, C_{m-1}}$; then $\ms{B}$ is a projective IFS and
$\ms{C}$ an affine one. Remember from~\S\ref{ref5} that we have surjective continuous maps $\pi_{\ms{B}},\pi_{\ms{C}}$ to the respective attractors~$[0,\lambda\m]$ and $[0,1]$ (see the vertical arrows in the ensuing~\eqref{eq7}, replacing $q$ in the middle row with $m-1$).
By standard arguments (see, e.g., the proof of~\cite[Theorem~9.2]{panti20b}), it is not difficult to show that $\pi_{\ms{B}}$ and $\pi_{\ms{C}}$ have the same fibers, so that the composition
$M_m=\pi_{\ms{C}}\circ\pi_{\ms{B}}\m$ is well defined and is an order-preserving homeomorphisms from $[0,\lambda\m]$ to $[0,1]$.
The map $M_3$ is the Minkowski question mark function, so we feel justified in naming any element of the family $\set{M_m}_{m\ge3}$
a \newword{Minkowski function}.
An equivalent definition is the following: let $P$ be the Bernoulli measure on $\set{1,\ldots,m-1}^{\Zbb\p}$ that assigns equal weight $1/(m-1)$ to each digit. Then $M$ is the cumulative distribution function of the pushforward measure $\mu=(\pi_{\ms{B}})_*P$ (of course, $(\pi_{\ms{C}})_*P$ is Lebesgue measure on $[0,1]$).

Remember from Remark~\ref{ref18} that $C_{\argomento}\circ(B_{\argomento})\m$ is an algebraic isomorphism between the submonoid $B_\Sigma$ of $\Gamma^\pm$ and the submonoid $C_\Sigma$ of $\Aff\Rbb$. Our next result shows that $M$ globally conjugates the action of $B_\Sigma$ on $[0,\lambda\m]$ to that of $C_\Sigma$ on $[0,1]$.

\begin{theorem}
\begin{enumerate}\label{ref19}
\item For every $\bm{s}\in\Sigma$ we have $M\circ B_{\bm{s}}=C_{\bm{s}}\circ M$.
\item Let $F$ be any Farey-type map as in Definition~\ref{ref6}. Then $T=M\circ F\circ M\m$ is a piecewise-affine selfmap of $[0,1]$ all of whose pieces have coefficients in $\Zbb[1/(m-1)]$.
\end{enumerate}
\end{theorem}
\begin{proof}
The statement~(2) is a straightforward consequence of~(1); we only present in~\eqref{eq7} the relevant commuting diagram, $\ms{B}$ and $\ms{C}$ being $\set{B_{\mb{s}_1},\ldots,B_{\mb{s}_q}}$ and $\set{C_{\mb{s}_1},\ldots,C_{\mb{s}_q}}$, and 
the middle arrow $S$ the shift map.
\begin{equation}\label{eq7}
\begin{tikzcd}
{[0,\lambda\m]} \arrow[dd, bend right=120,looseness=1.5, "M"'] \arrow[r,"F"] 
& {[0,\lambda\m]} \arrow[dd, bend left=120,looseness=1.5, "M"]  \\
\set{1,\ldots,q}^{\Zbb\p} \arrow[r,"S"] \arrow[u,"\pi_{\ms{B}}"'] \arrow[d,"\pi_{\ms{C}}"] & \set{1,\ldots,q}^{\Zbb\p} \arrow[u,"\pi_{\ms{B}}"] \arrow[d,"\pi_{\ms{C}}"'] \\
{[0,1]} \arrow[r,"T"] & {[0,1]}
\end{tikzcd}
\end{equation}

We prove~(1); since $B_{\argomento}$ and $C_{\argomento}$ are homomorphisms, it is sufficient to establish the identity for the generators $1,\ldots,m-1,f$ of $\Sigma$. Redefine $\ms{B}$ and $\ms{C}$ to be $\set{B_1,\ldots,B_{m-1}}$ and $\set{C_1,\ldots,C_{m-1}}$, respectively.
The key ingredient now is that both $\pi_{\ms{B}}$ and $\pi_{\ms{C}}$ respect the lexicographic ordering: up to countably many exceptions (namely those sequences $\omega=\omega(0)\omega(1)\cdots\in\set{1,\ldots,m-1}^{\Zbb\p}$ that end up in a tail of $1$ or of $m-1$) we have that $\omega$ is less than or equal to $\omega'$ in the lexicographic ordering~$\preceq$ if and only if $\pi_{\ms{B}}(\omega)\le\pi_{\ms{B}}(\omega')$ in $[0,\lambda\m]$
if and only if $\pi_{\ms{C}}(\omega)\le\pi_{\ms{C}}(\omega')$ in $[0,1]$.
This is easily proved by looking at the first index at which $\omega$ and $\omega'$ differ, and using the fact that all the maps $B_j$ and $C_j$ are order isomorphisms.

Fix then $x=\pi_{\ms{B}}(\omega)$, with $M(x)=\pi_{\ms{C}}(\omega)$. For 
$j\in\set{1,\ldots,m-1}$ we have
\begin{multline*}
M\bigl(B_j(x)\bigr)=\mu\bigl([0,B_j(x)]\bigr)=P\bigl(\pi_{\ms{B}}\m[0,B_j(x)]\bigr)\\
=P\bigl(\set{\omega':\pi_{\ms{B}}(\omega')\le\pi_{\ms{B}}(j\omega)}\bigr)=
P\bigl(\set{\omega':\omega'\preceq j\omega}\bigr)\\
=P\bigl(\set{\omega':\pi_{\ms{C}}(\omega')\le C_j(\pi_{\ms{C}}(\omega))}\bigr)\\
=P\bigl(\pi_{\ms{C}}\m[0,C_j(M(x))]\bigr)=\Leb\bigl([0,C_j(M(x))]\bigr)=C_j(M(x)),
\end{multline*}
as desired. We have now to show the analogous identity for the generator~$f$. To this purpose, consider the selfhomeomorphism $\phi$
of the Cantor space $\set{1,\ldots,m-1}^{\Zbb\p}$ induced by exchanging $j$ with $m-j$; clearly $\phi$ is an involution, reverses the lexicographic order, and leaves $P$ invariant.

\smallskip

\paragraph{\emph{Claim}} We have
$B_f(\pi_{\ms{B}}(\omega))=\pi_{\ms{B}}(\phi(\omega))$ and
$C_f(\pi_{\ms{C}}(\omega))=\pi_{\ms{C}}(\phi(\omega))$.

\paragraph{\emph{Proof of Claim}} We compute
\begin{multline*}
\set{B_f(\pi_{\mc{B}}(\omega))}=
\bigcap_{n\ge0} B_f B_{\omega(0)}\cdots B_{\omega(n-1)} [0,\lambda\m]\\
=\bigcap_{n\ge0} B_{\phi(\omega(0))}\cdots B_{\phi(\omega(n-1))} B_f
[0,\lambda\m]
=\set{\pi_{\ms{B}}(\phi(\omega))},
\end{multline*}
because $B_fB_j=B_{\phi(j)}B_f$
for every $j\in\set{1,\ldots,m-1}$, and
$B_f$ fixes $[0,\lambda\m]$ globally. The proof of the other statement is analogous.

\smallskip

Having proved our claim, we compute
\begin{multline*}
M(B_f(x))=\mu\bigl([0,B_f(\pi_{\ms{B}}(\omega))]\bigr)=
P\bigl(\pi_{\ms{B}}\m\set{x':x'\le\pi_{\ms{B}}(\phi(\omega))}\bigl)\\
=
P\bigl(\set{\omega':\pi_{\ms{B}}(\omega')\le\pi_{\ms{B}}(\phi(\omega))}\bigl)
=P\bigl(\set{\omega':\omega'\preceq\phi(\omega)}\bigr)\\
=
P\bigl(\phi\set{\omega':\omega'\preceq\phi(\omega)}\bigr)
=P\bigl(\set{\omega'':\omega\preceq\omega''}\bigr)=
P\bigl(\set{\omega'':\pi_{\ms{C}}(\omega)\le\pi_{\ms{C}}(\omega'')}\bigl)\\
=
P\bigr(\pi_{\ms{C}}\m\set{x:\pi_{\ms{C}}(\omega)\le x}\bigr)
=\Leb\bigr([\pi_{\ms{C}}(\omega),1]\bigr)=
1-M(x)=C_f(M(x)).
\end{multline*}

\end{proof}

\begin{remark}
Note that \emph{the same} Minkowski function $M_m$ conjugates \emph{every} Farey-type map $F$ resulting from a decorated $(m-1)$-ary tree with the corresponding tent map $T$. For example, the Romik map in Figure~\ref{fig3} left is conjugated by $M_3$ with the tent map with branches $C_{11}\m(x)=4x$, $C_{12f}\m(x)=-4x+2$, $C_2\m(x)=2x-1$. This is not to be confused with the setting in~\cite{bocalinden18}, in which the Romik map is conjugated with another tent map, the one with branches $3x$, $-3x+2$, $3x-2$.
\end{remark}

Farey-type maps as in Definition~\ref{ref6} are of the intermittent type~\cite{pomeaumanneville80}: Lebesgue-almost all points spend most of their time in neighborhoods of parabolic cycles. This fact forces the conjugating functions $M$ to be purely singular: for Lebesgue-almost all $x\in[0,\lambda\m]$, the derivative $M'(x)$ exists and has value~$0$. Salem proved in~\cite{salem43} that $M_3$ is H\"older continuous of exponent $\log(2)/\bigl(2\log(\tau)\bigr)$, where $\tau$ is the golden ratio; in the following Theorem~\ref{ref21} we will prove the analogous result for our family $M_m$.

A few preliminaries: we recall (see~\cite{jungers09} or
\cite{guglielmizennaro14} for a detailed presentation) that, given any finite set $\ms{D}$ of real square matrices of the same dimension, the \newword{joint spectral radius} of $\ms{D}$ is the number $\tilde\rho(\ms{D})$ defined by
\begin{equation}\label{eq10}
\tilde\rho(\ms{D})=\limsup_{n\to\infty}\max\set{
\norm{D}^{1/n}:\text{$D$ is a product of $n$ elements of $\ms{D}$}}.
\end{equation}
Here $\norm{D}$ is the operator norm of $D$ induced by some vector norm, whose choice is irrelevant. We will make use of both the spectral and the $\infty$-norm, so we recall their characterizations:
\begin{align*}
\norm{D}_2&=\text{the square root of the spectral radius of $D^TD$};\\
\norm{D}_\infty&=\text{the maximal $1$-norm of rows of $D$}.\\
\end{align*}
The following observation is surprisingly useful, so we state it as a lemma.

\begin{lemma}
Let\label{ref22} $\ms{D}=\set{D_1,\ldots,D_q}$ be any finite subset of $\Mat_{d\times d}\Rbb$. Assume $j,k\in\set{1,\ldots,q}$ (possibly $j=k$) are such that $D_k=D_j^T$ and $\norm{D_j}_2\ge\norm{D}_2$ for every $D\in\ms{D}$. Then $\tilde\rho(\ms{D})=\norm{D_j}_2$.
\end{lemma}
\begin{proof}
Using~\cite[Lemma~3.1]{daubechieslag}, we have
\[
\norm{D_j}_2=\bigl[\rho(D_kD_j)\bigr]^{1/2}\le
\tilde\rho(\ms{D})\le\max\set{\norm{D}_2:D\in\ms{D}}=\norm{D_j}_2.
\]
\end{proof}

Given $m=3,4,5,\ldots$ and letting
$\ms{A}=\set{A_1,\ldots,A_{m-1}}$, with $A_j$ defined in~\eqref{eq2},
we set
\begin{equation*}
\rho_m=\norm{A_{\floor{m/2}}}_2,\qquad
\alpha_m=\frac{\log(m-1)}{2\log(\rho_m)}.
\end{equation*}

\begin{theorem}
The\label{ref21} following statements are true.
\begin{enumerate}
\item $\tilde\rho(\ms{A})=\rho_m$;
\item $M_m$ is H\"older of exponent $\alpha_m$, and not H\"older of exponent $\beta$, for any $\beta>\alpha_m$.
\end{enumerate}
\end{theorem}
\begin{proof}
By Lemma~\ref{ref4}, for every $j\in\set{1,\ldots,m-1}$ we have $A_j^T=A_{\sharp j}=A_{m-j}$. Thus (1) follows from Lemma~\ref{ref22} once we prove that 
$A_{\floor{m/2}}$ has maximal spectral norm among matrices in $\ms{A}$.

Let $A$ be any matrix in $\SL_2\Rbb$; then $A$ has singular value decomposition $A=O_1\ppmatrix{s}{}{}{1/s}O_2$, with $O_1,O_2\in\SO_2\Rbb$ and $s=\norm{A}_2$. All orthogonal matrices fix $i\in\mc{H}$ and leave invariant the hyperbolic distance $d$. We thus obtain
\begin{multline*}
d(i,A(i))=d\bigl(O_1\m(i),\ppmatrix{s}{}{}{1/s}O_2(i)\bigr)\\
=
d(i,\ppmatrix{s}{}{}{1/s}(i))=d(i,s^2i)=2\log(s),
\end{multline*}
whence $\exp\bigl(-d(i,A(i))\bigr)=\norm{A}_2^{-2}$, an identity that will reappear in the final section.

It is sufficient to show that $d(i,A_{\floor{m/2}}(i))\ge
d(i,A_j(i))$, for every $j$. Since $A_j=SR^{-j}$ and $S$ is orthogonal, we have
as above $d(i,A_j(i))=d(i,R^{-j}(i))$.
The three points $\zeta,i,R^{-j}(i)$ are the vertices of a hyperbolic isosceles triangle (degenerate, if $m$ is even and $j=m/2$), with angle at $\zeta$ equal to $2j\pi/m$; see the dashed triangle in Figure~\ref{fig1} for the case $m=7$, $j=2$. By the hyperbolic cosine rule~\cite[\S7.12]{beardon95} we get
\[
\cosh\bigl(d(i,R^{-j}(i))\bigr)=\bigl[\cosh(d(i,\zeta))\bigr]^2-
\bigl[\sinh(d(i,\zeta))\bigr]^2 \cos(2j\pi/m),
\]
which is maximized by $j=\floor{m/2}$, as required.

In order to prove~(2) we recall the well known group embedding of $\PSL_2\Rbb$ into $\SO_{2,1}\Rbb$:
\[
\begin{pmatrix}
a & b\\
c & d
\end{pmatrix}
\mapsto
\begin{pmatrix}
ad+bc & ac-bd & ac+bd \\
ab-cd & (a^2-b^2-c^2+d^2)/2 & (a^2+b^2-c^2-d^2)/2 \\
ab+cd & (a^2-b^2+c^2-d^2)/2 & (a^2+b^2+c^2+d^2)/2
\end{pmatrix}.
\]

\smallskip

\paragraph{\emph{Claim}}
\begin{itemize}
\item[(i)] Letting $\mb{A}$ be the image of $A$ under the above embedding, we have $\tilde\rho\bigl(\set{\mb{A}_1,\ldots,\mb{A}_{m-1}}\bigr)=
\tilde\rho(\ms{A})^2=
\rho_m^2$.
\item[(ii)] Let $\mb{j}=j_0\cdots j_{n-1}\in\set{1,\ldots,m-1}^*$,
and let $J_{\mb{j}}=B_{j_0}\cdots B_{j_{n-1}}[0,\lambda\m]$ be the corresponding cylinder. Then
$8\m\norm{\mb{A}_{j_0}\cdots\mb{A}_{j_{n-1}}}_\infty\m<\length(J_{\mb{j}})$.
\end{itemize}

\paragraph{\emph{Proof of Claim}} Let $\rho(D)$ be the spectral radius of the matrix $D$.

(i) This follows from the
Berger-Wang characterization~\cite{bergerwang92} of the joint spectral radius
\begin{equation*}
\tilde\rho(\ms{D})=\limsup_{n\to\infty}\max\set{
\rho(D)^{1/n}:\text{$D$ is a product of $n$ elements of $\ms{D}$}},
\end{equation*}
and the identity
$\rho(\mb{A})=\rho(A)^2$, which is valid for every $A\in\PSL_2\Rbb$.
Indeed $A$ is conjugate to a matrix either of the form
\[
\begin{pmatrix}
\cos(t) & -\sin(t) \\
\sin(t) & \cos(t)
\end{pmatrix},\text{ or }
\begin{pmatrix}
\exp(t/2) &  \\
 & \exp(-t/2)
\end{pmatrix},\text{ or }
\begin{pmatrix}
1 & t \\
 & 1
\end{pmatrix},
\]
and for such matrices the identity follows by direct inspection (see equations~(2) in~\cite{panti20b}).

(ii) Let
\[
A=\begin{pmatrix}
a & b \\
c & d
\end{pmatrix}
=A_{j_0}\cdots A_{j_{n-1}},\quad
D=\begin{pmatrix}
1 &  \\
\lambda & 1
\end{pmatrix}A.
\]
Then the columns of $D$ give projective coordinates for the endpoints of
$J_{\mb{j}}=[b/(\lambda b+d),a/(\lambda a+c)]$. Moreover,
$a,b,c,d$ are all nonnegative, and
we have
\begin{multline*}
8\norm{\mb{A}}_\infty\ge
8\norm{\text{$3$rd row of $\mb{A}$}}_1=
8\bigl(ab+cd+\max\set{a^2+c^2,b^2+d^2}\bigr)\\
>\max\bigl\{(2a+c)^2,(2b+d)^2\bigr\}
>(\lambda a+c)(\lambda b+d)=1/\length(J_{\mb{j}}),
\end{multline*}
which settles our claim.

\smallskip

Fix now $\varepsilon>0$ and let $\rho=\rho_m$, $\alpha=\alpha_m$,
$\theta=\rho^2+\varepsilon$. 
By~(i) and~\eqref{eq10} there exists $n_1$ such that, for every word~$\mb{j}$ of length $\ge n_1$, we have
$\theta>\norm{\mb{A}_{\mb{j}}}_{\infty}^{1/\length(\mb{j})}$.
Let $l_1$ be the minimal length of cylinders of level $n_1$, and
choose $0\le x<x'\le\lambda\m$ such that $x'-x\le l_1$. Also, let $n\ge n_1$ be minimum such that $[x,x']$ contains a cylinder of level $n$, say $J_{\mb{h}}$. By~(ii) we have
$x'-x>8\m\norm{\mb{A}_{\mb{h}}}_{\infty}\m>8\m\theta^{-n}$, and thus
\begin{equation}\label{eq8}
n>-\log_\theta(x'-x)-\log_\theta(8).
\end{equation}
On the other hand, $[x,x']$ contains at most $1+2(m-2)$ endpoints of cylinders of level $n$; thus
$Mx'-Mx<2(m-1)(m-1)^{-n}$ and
\begin{equation}\label{eq9}
n<1+\log_{m-1}(2)-\log_{m-1}(Mx'-Mx).
\end{equation}
Eliminating $n$ from~\eqref{eq8} and~\eqref{eq9} and rearranging terms we obtain
\[
\log(Mx'-Mx)<\log(C)+\log(x'-x)\frac{\log(m-1)}{\log(\theta)},
\]
for some constant $C>2(m-1)$. We let $\varepsilon$ tend to $0$ and obtain
$Mx'-Mx\le C(x'-x)^\alpha$, the requested H\"older condition for all pairs $x<x'$ at distance $\le l_1$. Replacing $C$ with $\max\set{C,l_1^{-\alpha}}$ gives the condition for all pairs in $[0,\lambda\m]$.

It remains to show the optimality of $\alpha$. Let
$H=A_{m-\floor{m/2}}A_{\floor{m/2}}$; then the columns of
\[
\begin{pmatrix}
1 & \\
\lambda & 1
\end{pmatrix}H^n=
\begin{pmatrix}
1 & \\
\lambda & 1
\end{pmatrix}
\begin{pmatrix}
a_n & b_n\\
c_n & d_n
\end{pmatrix}=
\begin{pmatrix}
1 & \\
\lambda & 1
\end{pmatrix}E
\begin{pmatrix}
\rho^{2n} & \\
 & \rho^{-2n}
\end{pmatrix}E\m,
\]
where $E$ is a diagonalizing matrix, give the projective coordinates of the endpoints $x_n,x'_n$ of a cylinder of level $2n$ and length
$\bigl((\lambda a_n+c_n)(\lambda b_n+d_n)\bigr)\m$.
A straightforward computation shows that this length is asymptotic to
$C\rho^ {-4n}$, for a constant~$C$ depending on $\lambda$ and the entries of $E$. But then, for every $\beta=\alpha+\varepsilon>\alpha$, we have
\[
\lim_{n\to\infty}\frac{Mx_n'-Mx_n}{(x_n'-x_n)^\beta}=
\lim_{n\to\infty}\frac{(m-1)^{-2n}}{C^\beta \rho^{-4n\alpha}\rho^{-4n\varepsilon}}=
\lim_{n\to\infty}\frac{(m-1)^{-2n}}{C^\beta (m-1)^{-2n} \rho^{-4n\varepsilon}}=
\infty;
\]
thus $M$ is not H\"older of exponent~$\beta$.
\end{proof}

\section{End of proof of Theorem~\ref{ref8}}\label{ref12}

We can now conclude the proof of Theorem~\ref{ref8} by establishing~(1).
Let us recall the setting: $\mb{s}_1,\ldots,\mb{s}_q$ are the leaves of a decorated $(m-1)$-ary tree, and $K\subseteq[0,\lambda\m]$ is the attractor
of $\sharp\ms{B}=\set{B_{\sharp\mb{s}_1},\ldots,B_{\sharp\mb{s}_q}}$.
By Theorem~\ref{ref19}(1) (whose proof does not depend on Theorem~\ref{ref8}(1)), the Minkowski homeomorphism $M:[0,\lambda\m]\to[0,1]$ satisfies $M\circ B_{\sharp\mb{s}_i}\circ M\m=C_{\sharp\mb{s}_i}$ for all $i$. As a consequence $K$ is the $M\m$-image of the attractor $K_1$ of $\sharp\ms{C}=\set{C_{\sharp\mb{s}_1},\ldots,C_{\sharp\mb{s}_q}}$.

Let $\mb{s}=j_0\cdots j_{l-1}f^e$ be any element of $\Sigma$, with $e\in\set{0,1}$; note that $l=l(\mb{s})$ is the level of $\mb{s}$ in the $(m-1)$-ary tree.
As an affine map, $C_{\mb{s}}(x)$ equals $(-1)^e(m-1)^{-l}x+h(m-1)^{-l}$, for some integer $0\le h\le(m-1)^l$. The subgroup $G_\Sigma$ of $\Aff\Rbb$ generated by the monoid $C_\Sigma$ is endowed with the topology of pointwise convergence, which is the same as the topology of convergence over the two-point set $\set{0,1}$. More simply ---and equivalently--- we will use the euclidean topology induced by the embedding $G_\Sigma\ni(ax+b)\mapsto(a,b)\in\Rbb^*\times\Rbb$.

Now, although $G_\Sigma$ is not discrete (for example, $C_{0^l}C_{0^{l-1}1}\m(x)=x-(m-1)^{-l}$), the following fact holds.

\smallskip

\paragraph{\emph{Claim}} The identity function $x$ does not belong to the closure of the set
$\set{C\m D:C,D\in C_\Sigma\text{ and }C\not=D}$.

\paragraph{\emph{Proof of Claim}} The slope of the composition of two affine functions is the product of the slopes of the factors; every element in $G_\Sigma$ has thus slope in $\pm(m-1)^\Zbb$. In particular, $C\m D$ may approach $x$ only if $C$ and $D$ have the same slope.
We thus must have $C(x)=(-1)^e(m-1)^{-l} x+h(m-1)^{-l}$ and
$D(x)=(-1)^e(m-1)^{-l} x+k(m-1)^{-l}$, with $h,k$ different integers.
But then $C\m D(x)=x+(-1)^e(k-h)$, which is plainly isolated from the identity.

\smallskip

From the definition of a decorated tree it is readily seen that 
$\sum_{1\le i\le q}(m-1)^{-l(\mb{s}_i)}=1$. Since $l(\sharp\mb{s}_i)=l(\mb{s}_i)$ and $(m-1)^{-l(\sharp\mb{s}_i)}$ is the contraction factor of $C_{\sharp\mb{s}_i}$, the similarity dimension of $K_1$
is~$1$~\cite[p.714]{hutchinson81}.

As in the proof of Lemma~\ref{ref4}, by the Ping-Pong Lemma the functions
$C_{\mb{s}_1},\ldots,C_{\mb{s}_q}$ generate a free monoid. 
Again by Lemma~\ref{ref4}, the elements $\mb{s}_1,\ldots,\mb{s}_q$ generate a free submonoid of $\Sigma$, and so do their antiisomorphic images
$\sharp\mb{s}_1,\ldots,\sharp\mb{s}_q$; therefore $C_{\sharp\mb{s}_1},\ldots,C_{\sharp\mb{s}_q}$ generate a free monoid as well.
Given a word $u=\sharp\mb{s}_{i(0)}\cdots\sharp\mb{s}_{i(t-1)}\in
\set{\sharp\mb{s}_1,\ldots,\sharp\mb{s}_q}^*$, we abbreviate
$C_{\sharp\mb{s}_{i(0)}}\cdots C_{\sharp\mb{s}_{i(t-1)}}$ by $C_u$. Then, by
the above claim and the freeness of the monoid generated by $C_{\sharp\mb{s}_1},\ldots,C_{\sharp\mb{s}_q}$, the identity function does not belong to the closure of the set of all products $C_u\m C_v$, where $u,v$ are different words in
$\set{\sharp\mb{s}_1,\ldots,\sharp\mb{s}_q}^*$.
By the main result of~\cite{bandtgraf92}, $K_1$ has positive Hausdorff measure in its similarity dimension; in our case this means positive Lebesgue measure.
By~\cite[Theorem~2.2 and Corollary~2.3]{schief94},
$\sharp\ms{C}$ satisfies the strong open set condition w.r.t.~an open set $V_1$ whose closure equals $K_1$; since $M$ is a homeomorphism, by taking $V=M\m[V_1]$ we obtain (1.1) and (1.2) of Theorem~\ref{ref8}.
As the sum of the contraction ratios of the elements of $\sharp\ms{C}$ is~$1$ and the images of $V_1$ are disjoint, we surely have $\Leb\bigl(V_1\setminus\Phi_{\sharp\ms{C}}(V_1)\bigr)=0$. However, since~$M$ is a singular function, the corresponding statement for $V$ ---which is precisely Theorem~\ref{ref8}(1.3)--- is not immediate. 

We let then, as in the proof of~\cite[Corollary~1.2]{peres_et_al01},
$W=V\setminus\Phi_{\sharp\ms{B}}(V)$ and observe that for $u,v\in\set{\sharp\mb{s}_1,\ldots,\sharp\mb{s}_q}^*$ different words as above, $B_u[W]\cap B_v[W]=\emptyset$; therefore
\begin{equation}\label{eq11}
\sum\bigl\{\Leb(B_u[W]):
u\in\set{\sharp\mb{s}_1,\ldots,\sharp\mb{s}_q}^*\bigr\}
\le\Leb(V)<\infty.
\end{equation}
Say $B_{u}=\ppmatrix{a}{b}{c}{d}$; then the maximum and minimum absolute values of the derivative of $B_{u}$ in $[0,\lambda\m]$ are achieved at the interval endpoints, and thus are $B_{u}'(0)=d^{-2}$ and $B_{u}'(\lambda\m)=(\lambda\m c+d)^{-2}$.
Letting
\begin{equation*}
r_{u}=\min\set{d^{-2},(\lambda\m c+d)^{-2}}=
\bigl[\max\set{\abs{d},\abs{\lambda\m c+d}}\bigr]^{-2},
\end{equation*}
we have $\Leb(B_u[W])\ge r_u\Leb(W)$.
If we can establish
\begin{equation}\label{eq12}
\sum\bigl\{r_u:u\in\set{\sharp\mb{s}_1,
\ldots,\sharp\mb{s}_q}^*\bigr\}=\infty,
\end{equation}
then~\eqref{eq11} will imply $\Leb(W)=0$, as required.

Let $\vertiii{\argomento}$ be the vector norm on $\Rbb^2$ given by $\vertiii{\cppvector{x}{y}}=\bigl\Vert\ppmatrix{\lambda\m}{}{}{1}\cppvector{x}{y}\bigr\Vert_\infty$, and note that every
$u\in\set{\sharp\mb{s}_1,\ldots,\sharp\mb{s}_q}^*$, viewed as an element of $\Sigma$, is the $\sharp$-image of a unique $v\in\set{\mb{s}_1,\ldots,
\mb{s}_q}^*$. From Lemma~\ref{ref4} we have $A_u^T=A_v$, and therefore
\begin{align*}
r_u&=\biggl\Vert\biggl(
\begin{pmatrix}
0 & 1
\end{pmatrix}
B_{u}
\begin{pmatrix}
\lambda\m & \\
1 & 1
\end{pmatrix}
\biggr)^T\biggr\Vert_\infty^{-2}\\
&=\biggl\Vert\biggl(
\begin{pmatrix}
0 & 1
\end{pmatrix}
\begin{pmatrix}
1 & \\
\lambda & 1
\end{pmatrix}
A_{u}
\begin{pmatrix}
1 & \\
-\lambda & 1
\end{pmatrix}
\begin{pmatrix}
\lambda\m & \\
1 & 1
\end{pmatrix}
\biggr)^T\biggr\Vert_\infty^{-2}\\
&=\biggl\Vert
\begin{pmatrix}
\lambda\m & \\
 & 1
\end{pmatrix}
A_{v}
\begin{pmatrix}
\lambda \\
1
\end{pmatrix}
\biggr\Vert_\infty^{-2}\\
&=\vertiii{
A_{v}
\begin{pmatrix}
\lambda \\
1
\end{pmatrix}}^{-2}.
\end{align*}
Thus~\eqref{eq12} reduces to
\begin{equation}\label{eq13}
\sum\biggl\{\vertiii{A_{v}
\begin{pmatrix}
\lambda \\
1
\end{pmatrix}}^{-2}:
v\in\set{\mb{s}_1,\ldots,\mb{s}_q}^*\biggr\}=\infty.
\end{equation}
All vector norms are equivalent, so we safely replace $\vertiii{\argomento}$ with the spectral norm $\norm{\argomento}_2$; also, as noted in the proof of Theorem~\ref{ref21},
\[
\biggl\lVert A_{v}
\begin{pmatrix}
\lambda \\
1
\end{pmatrix}\biggr\rVert_2^{-2}\ge
\norm{A_{v}}_2^{-2}
\biggl\lVert
\begin{pmatrix}
\lambda \\
1
\end{pmatrix}\biggr\rVert_2^{-2}
=
\exp\bigl(-d(i,A_{v}(i))\bigr)\,
\biggl\lVert
\begin{pmatrix}
\lambda \\
1
\end{pmatrix}\biggr\rVert_2^{-2}.
\]
Thus~\eqref{eq13} reduces again to the divergence of the Poincar\'e series
\begin{equation}\label{eq14}
\sum\bigl\{\exp(-d(i,A_{v}(i))\bigr):
v\in\set{\mb{s}_1,\ldots,\mb{s}_q}^*\bigr\}.
\end{equation}

Now, since $\Gamma$ is finitely generated, it is of divergence type (see~\cite[Chapter~1]{nicholls89} for a nice introduction to these topics). Therefore the series in~\eqref{eq14} would surely diverge if $A_{v}$ were to range over all of $\Gamma^\pm$. We take advantage of the fact, mentioned in~\S\ref{ref2}, that $\Gamma^\pm$ is an amalgamated product of dihedral groups; this implies that every element of $\Gamma^\pm$ can be uniquely written as a \newword{reduced word} in the generators $S,R,F$, the adjective ``reduced'' referring to the restrictions:
\begin{enumerate}
\item $F$ may occur at most once, and precisely at the word end;
\item neither $S^2$ nor $R^m$ occur.
\end{enumerate}
Let $\mc{Q}$ be the set of reduced words that do not begin with $R$; from the above mentioned divergence one easily sees that
\begin{equation}\label{eq15}
\sum\bigl\{\norm{A}_2^{-2}:
A\in\mc{Q}\bigr\}=\infty.
\end{equation}

Each $A_{\mb{s}_i}$ is a product of $A_1=SR^{m-1}$, $A_2=SR^{m-2}$, $\ldots$, $A_{m-1}=SR$, followed by an eventual $F$; therefore it is already in reduced form, belongs to $\mc{Q}$, begins with $S$ and ends either with $R^j$ or with $R^jF$. An \newword{initial segment} of $A_{\mb{s}_i}$ is any reduced word that can be obtained from $A_{\mb{s}_i}$ by first deleting the final $R^j$ or $R^jF$, and then recursively deleting any final $S$ or final $R^h$ (with $h$ maximal), until the empty word~$\emptyset$ is reached.
Let $\mc{Q}'$ be the finite set that contains every initial segment of every $A_{\mb{s}_i}$ (for $1\le i\le q$), as well as every word obtainable from these segments by adding a final $F$.

\begin{example}
Consider\label{ref23} the rightmost decorated tree in Figure~\ref{fig2}; here $m=4$.
We have
\begin{gather*}
A_{\mb{s}_1}=A_1=SR^3,\,
A_{\mb{s}_2}=A_2 A_1 A_f=SR^2SR^3F,\\
A_{\mb{s}_3}=A_2^2A_f=SR^2SR^2F,\,
A_{\mb{s}_4}=A_2A_3=SR^2SR,\,
A_{\mb{s}_5}=A_3A_f=SRF.
\end{gather*}
The set of initial segments of $A_{\mb{s}_2}$ is then $\set{SR^2S,SR^2,S,\emptyset}$, and $\mc{Q}'=\set{S,\emptyset,SR^2S,SR^2,SF,F,SR^2SF,SR^2F}$.
\end{example}

The situation is clarified by looking at the Caley graph of $\Gamma^\pm$ over the generators $S,R,F$. Caley graphs for the extended Hecke groups are made of two identical layers connected by vertical $F$-edges. Each layer has the structure of an infinite $m$-ary tree, with nodes replaced by $m$-long
circuits of $R$-edges, oriented clockwise in the upper layer and counterclockwise in the lower one. If we look at the portion of the graph that lies between the identity~$\emptyset$ and the leaves $A_{\mb{s}_1},\ldots,A_{\mb{s}_q}$, we readily realize that $\mc{Q}'$ is just the set of nodes in this finite portion which are different from the leaves and not connected to the leaves by an $F$-edge. See Figure~\ref{fig8} for the case in Example~\ref{ref23}; a few nodes are labelled, the black dots in the upper layer correspond to the undecorated leaves of the tree, while the white dots in the lower layer to the decorated ones.

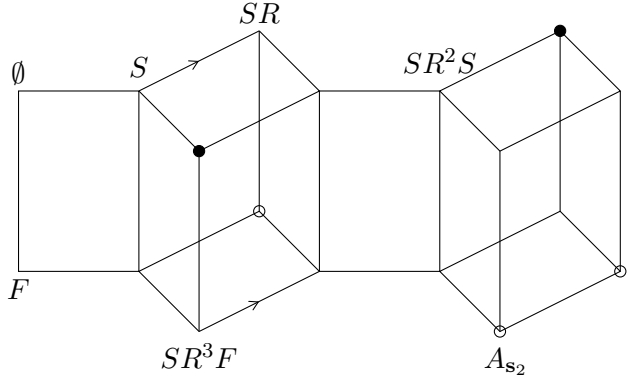
\begin{figure}[h!]
\begin{tikzpicture}[scale=0.8]
\node at (0,4.3) [] {$\emptyset$};
\node at (0,0.7) [] {$F$};
\node at (2,4.4) [] {$S$};
\node at (4,5.3) [] {$SR$};
\node at (3,-0.4) [] {$SR^3F$};
\node at (7,4.5) [] {$SR^2S$};
\node at (8.1,-0.5) [] {$A_{\mb{s}_2}$};
\node at (3,3) [vertex2]  {};
\node at (4,2) [vertex1]  {};
\node at (9,5) [vertex2]  {};
\node at (8,0) [vertex1]  {};
\node at (10,1) [vertex1]  {};
\draw[middlearrow={angle 60}] (2,4)--(4,5);
\draw (4,5)--(5,4)--(3,3)--(2,4);
\draw (0,4)--(2,4)--(2,1)--(0,1)--cycle;
\draw[middlearrow={angle 60}] (3,0)--(5,1);
\draw (5,1)--(4,2)--(2,1)--(3,0);
\draw (4,5)--(4,2);
\draw (3,3)--(3,0);
\draw (5,4)--(7,4)--(7,1)--(5,1)--cycle;
\draw (7,4)--(9,5)--(10,4)--(8,3)--cycle;
\draw (7,1)--(9,2)--(10,1)--(8,0)--cycle;
\draw (9,5)--(9,2);
\draw (8,3)--(8,0);
\draw (10,4)--(10,1);
\end{tikzpicture}
\caption{A portion of the Cayley graph of $\Gamma_4^\pm$}
\label{fig8}
\end{figure}

A brief pondering on Figure~\ref{fig8} reveals that every $A\in\mc{Q}$ factors uniquely as $A=A'Q'$, with $A'$ a product of $A_{\mb{s}_1},\ldots,A_{\mb{s}_q}$ and $Q'\in\mc{Q}'$ (note that $A'Q'$ is not necessarily reduced).
Let $c=\min\bigl\{\norm{(Q')\m}_2^{-2}:Q'\in\mc{Q}'\bigr\}$. Since norms are submultiplicative, from $A'=A(Q')\m$ we obtain $\norm{A'}_2^{-2}\ge c\norm{A}_2^{-2}$, and this implies that the series in~\eqref{eq14} is bounded from below by
\[
\frac{c}{\mathrm{card}(\mc{Q}')}
\sum\bigl\{\norm{A}_2^{-2}:A\in\mc{Q}\bigr\}.
\]
The equality in~\eqref{eq15} shows then that the series in~\eqref{eq14} diverges, thus completing the proof of Theorem~\ref{ref8}.


\end{document}